\documentclass[a4paper,10pt]{article}
\usepackage[centertags]{amsmath}
\usepackage{amsfonts}
\usepackage{amssymb}
\usepackage{amsthm}
\usepackage{epsfig}
\usepackage{setspace}
\usepackage{ae}
\usepackage{eucal}
\usepackage[usenames]{color}%
\usepackage{graphicx}

\theoremstyle{plain}
\newtheorem{thm}{Theorem}[section]
\newtheorem{lem}[thm]{Lemma}
\newtheorem{cor}[thm]{Corollary}

\theoremstyle{definition}
\newtheorem{defi}[thm]{Definition}

\theoremstyle{remark}

\newtheorem{rem}[thm]{Remark}

\renewcommand{\div}{\operatorname{div}}

\title{Some new consequences of the CKN-theory}
\author{J\"org Kampen }

\begin{document}

\maketitle

\begin{abstract}
It is a simple consequence of the Cafarelli-Kohn-Nirenberg theory that every possible singularity in a thin Haussdorff-measurable set of a Leray-Hopf solution of the incompressible  Navier Stokes equation is on the tip of a small open cone, where the solution is smooth. Using global regularity results for weak Hopf-Leray solutions this potential singularity can be analyzed by investigation of the asymptotic behavior at infinite time of a solution of a related initial-boundary value problem posed in transformed coordinates on a cylinder. It follows that the velocity components and their spatial first order derivatives are left continuous at each potential singular point. Next to some new consequences such as as global regularity of the Leray Hopf solution after finite time (for $L^2$-data) many known results can be recovered with this method succinctly, for example the result that $H^1$-regularity implies global existence and smoothness, and that local regularity follows from $L^3$-regularity of the data.  
\end{abstract}

\section{Some simple observations about Leray-Hopf solutions in the light of the CKN-theorem}

Hopf told us essentially that for $L^2$-data $\mathbf{h}=(h_1,\cdots ,h_n)$ and positive viscosity $\nu >0$ the Navier-Stokes equation initial value problem with periodic boundary conditions, i.e., the problem 

\begin{equation}\label{Nav}
\left\lbrace \begin{array}{ll}
\frac{\partial v_i}{\partial t}-\nu \sum_{j=1}^n v_{i,j,j}
+\sum_{j=1}^n v_jv_{i,j}=-\nabla_i p,\\
\\
\mbox{div}~\mathbf{v}=0,\\
\\
\mathbf{v}(0,.)=\mathbf{h},
\end{array}\right.
\end{equation}
has a global weak solution for the velocity components $v_i,~1\leq i\leq n$, where
\begin{equation}\label{Hopf}
v_i\in L^{\infty}\left({\mathbb R}_+,L^2\left({\mathbb T}^n\right)  \right)\cap L^2_{loc}\left( {\mathbb R}_+,H^1\left({\mathbb T}^n\right) \right).  
\end{equation}
The pressure $p$ is determined by the velocity components via the Leray projection. Here, the symbol ${\mathbb T}^n$ denotes the torus of dimension $n$ and  ${\mathbb R}_+$ denotes the set of positive real numbers. The arguments in \cite{CKN}, or in \cite{L}, tell us in addition that 
\begin{itemize}
 \item[i)] there exists an open dense subset $I\subset (0,\infty)$ such that $m_L\left( {\mathbb R}_+\setminus I\right)=0$, where $m_L$ denotes the Lebesgue measure, and such that
 \begin{equation}
 v_i\in C^{\infty}\left(I\times {\mathbb T}^n\right),~1\leq i\leq n, 
 \end{equation}
where $C^{\infty}$ denotes the space of smooth functions.
 \item[ii)]  \begin{equation}
             S:=\left\lbrace (t,x)\in {\mathbb R}_+\times {\mathbb T}^n|v_i\not\in C^{\infty}
	     \mbox{ at }(t,x)
	      \mbox{ for some $1\leq i\leq n$}
	      \right\rbrace 
             \end{equation}
has vanishing one dimensional Hausdorff measure.
\item[iii)] for any $T>0$ and for dimension $n=3$ a weak Leray-Hopf solution $v_i,~1\leq i\leq 3$ satisfies
\begin{equation}\label{n3eq}
v_i\in L^{\frac{8}{3}}\left(\left[0,T\right],L^4\left({\mathbb T}^3\right)   \right) 
\end{equation}
for all $1\leq i\leq n$.
\end{itemize}
Proofs of these results can also be found in standard text books such as volume III of Taylor's book on partial differential equations. 
The open dense set $I$ may be represented by an union of open intervals, i.e.,
\begin{equation}
I=\cup_{j\in J}U_j,
\end{equation}
where $J$ is an index set and $U_j\subset {\mathbb R}_+$ are open intervals in the field of real numbers ${\mathbb R}$ equipped with the standard topology. Next if $t_s\in {\mathbb R}_+\setminus I$ is a time of a time slice $\left\lbrace t_s\right\rbrace\times {\mathbb T}^n$ related to a possible singularity, then there exists an index $j\in J$ and an open interval $U_j=(t_0,t_s)\subset I$. According to \cite{CKN} we know
\begin{equation}
v_i\in C^{\infty}\left(U_j\times {\mathbb T}^n\right)~\mbox{ for }1\leq i\leq n, 
\end{equation}
while (\ref{n3eq}) implies
\begin{equation}\label{83}
v_i\in L^{\frac{8}{3}}\left(\left[t_0,t_s\right],L^4\left({\mathbb T}^3\right)   \right), 
\mbox{ and }
v_i\in L^{\frac{8}{3}}\left(U_j,L^4\left({\mathbb T}^3\right)   \right)
\end{equation} 
especially. The gap to a proof of uniqueness  is filled if for any given finite horizon $T>0$ we can prove that a Leray-Hopf solution $v_i,¸1\leq i\leq n$ satisfies
\begin{itemize}
 \item[i)] $v_i\in L^{\infty}\left({\mathbb R}_+,L^2\left({\mathbb T}^n\right)  \right)\cap L^2\left( [0,T],H^1\left({\mathbb T}^n\right) \right)$;
 \item[ii)] $v_i\in L^{8}\left(\left[0,T\right],L^4\left({\mathbb T}^3\right)   \right)$.
\end{itemize}
From the perspective of weak function spaces it seems difficult to close the gap between item ii) and (\ref{83}), while the condition in item i) is quite close to what is known by Hopf's result.   
In this context recall that in the case of dimension $n=3$ the standard (mollified equation) arguments for Hopf's result in (\ref{Hopf}) use compact sequences with limit $v_i\in L^2\left([0,T], H^{1-\epsilon}\right)$ for any $\epsilon >0$ and $T>0$. Indeed, these standard arguments tell us that a family of solutions $v^{\epsilon_k}_i,~1\leq i\leq n,~k\geq 1$ of mollified equations build a compact sequence in $L^2\left([0,T],H^{1-\epsilon_0} \right)$ for arbitrarily small $\epsilon_0 >0$ such that (passing to a subsequence denoted again by $(v^{\epsilon_k}_i)_{k\geq 1},~1\leq i\leq n$ if necessary) the set
\begin{equation}
\left\lbrace t\in [0,T]|\forall~1\leq i\leq n~{\big |}v^{\epsilon_1}_i(t,.){\big |}_{H^{1-\epsilon_0}}+ \sum_{k\geq 2}{\big |}v^{\epsilon_k}_i(t,.)-v^{\epsilon_{k-1}}_i(t,.){\big |}_{H^{1-\epsilon_0}}<\infty \right\rbrace 
\end{equation}
is dense in $[0,T]$. This leads to the open dense time set where the weak Leray-Hopf solution is smooth, such that any possible singularity after any finite time is at the endpoint of an open interval $\left(t_0,t_0+\Delta_{t_0} \right)$ of smoothness (for some $\Delta_{t_0}>0$). More precisely, for a zero sequence $(\epsilon_k)_k$ the weak limit of $(v^{\epsilon_k}_i)_{k\geq 1},~1\leq i\leq n$ satisfies
\begin{equation}
v_i=\lim_{k\uparrow\infty}v^{\epsilon_k}_i\in L^2\left([t_0,t_0+\Delta_{t_0}],H^{1-\epsilon_0} \right)\cap C^{\infty}\left((t_0,t_0+\Delta_{t_0})\times {\mathbb T}^n \right).
\end{equation}
Hence possible singularities at $(t_s,x_s)$ for positive time $t_s>0$ are at the larger endpoint of such an open time interval (as is well-known).

Next for given $t_s>0$ assume that $(t_s,x_s)\in S$ is a singular point of a Leray-Hopf solution. For given $j\in J$ consider the corresponding open cone
\begin{equation}
K^{(t_s,x_s)}_j:=\left\lbrace(t,x)|t\in U_j=(t_0,t_s)~\&~|x-x_s|< t_s-t\right\rbrace ,
\end{equation}
where $|.|$ denotes the Euclidean distance. For this $j$ and for all $1\leq i\leq n$ we have $v_i\in C^{\infty}\left(K^{(t_s,x_s)}_j \right)$ and
\begin{equation}\label{sing}
\mbox{ for all }~(t,x)\in \overline{K^{(t_s,x_s)}_j}:~|v_i(t,x)|\leq \frac{c}{(t_s-t)^{\mu}|x-x_s|^{\lambda}}
\end{equation}
for some finite constant $c\in {\mathbb R}_+$ and for some parameters $\mu,\lambda$ which satisfy (case $n=3$)
\begin{equation}\label{para}
0\leq \mu <\frac{3}{8},~0\leq \lambda < \frac{3}{4}.
\end{equation}
Concerning first order spatial derivatives we have for all $1\leq i\leq n$  $v_i\in L^2([0,T],H^{1-\epsilon})$ for small $\epsilon >0$ (by the standard mollified equation arguments), hence 
\begin{equation}\label{sing2}
~\mbox{ for all }~(t,x)\in \overline{K^{(t_s,x_s)}_j}:~{\big |}v_{i,k}(t,x){\big |}\leq \frac{c}{(t_s-t)^{\mu_0}|x-x_s|^{\lambda_0}}
\end{equation}
for some finite constant $c\in {\mathbb R}_+$, and for some parameters $\mu_0,\lambda_0$ which satisfy
\begin{equation}\label{para2}
0\leq \mu_0 <\frac{1}{2},~0\leq \lambda_0 < \frac{3}{2}+\epsilon \mbox{ for small $\epsilon >0$}.
\end{equation}
Here, $\overline{K^{(t_s,x_s)}_j}$ denotes the closure of the open cone $K^{(t_s,x_s)}_j$, i.e.,
\begin{equation}
\overline{K^{(t_s,x_s)}_j}:=\left\lbrace(t,x)|t\in [t_0,t_s]~\&~|x-x_s|\leq  t_s-t\right\rbrace .
\end{equation}
Next we introduce a technique to push the singulatity orders $\mu$ and $\lambda$ of the velocity components $v_i, 1\leq i\leq n$ from $\frac{3}{8}$ and $\frac{3}{4}$ respectively to zero and, similarly, the singularity orders $\mu_0$ and $\lambda_0$ of the first order spatial derivatives $v_{i,j},~ 1\leq i,j\leq n$ from $\frac{1}{2}$ and $\frac{3}{2}$ respectively to zero. This implies that the velocity components and the first order spatial derivatives of the velocity components of a Leray-Hopf solution are left continuous at each potential singular point $(t,x)\in S$. Since there may be singularities in the slice $\{t_0\}\times {\mathbb T}^n$, we consider the restricted domain
\begin{equation}
K^{(t_s,x_s)}_{j,t_1}:=\left\lbrace(t,x)|t\in (t_1,t_s)~\&~t_1\in U_j~|x-x_s|< t_s-t\right\rbrace .
\end{equation}
Then for all parameters which satisfy (\ref{para2}) the functions $u^{\lambda,\mu}_i:K^{(t_s,x_s)}_{j,t_1}\rightarrow {\mathbb R}, 1\leq i\leq n$ defined by
\begin{equation}\label{wlm}
u^{\lambda,\mu}_i(t,x)=(t_s-t)^{\mu}|x-x_s|^{\lambda}v_i(t,x)
\end{equation}
and their first order spatial derivatives 
$u^{\lambda,\mu}_{i,j}, 1\leq i,j\leq n$ 
are bounded functions on the closed cone $\overline{K^{(t_s,x_s)}_{j,t_1}}$, where 
$$u_i\in C^{\infty}\left(K^{(t_s,x_s)}_{j,t_1}\right),~1\leq i\leq n.$$ 

Next we define classes of singularities. 
\begin{defi}
A function $f: [0,T]\times {\mathbb R}^n\rightarrow {\mathbb R}$ is said to have an isolated singularity from the left, i.e., with respect to increasing time, if there is an open cone $K^{(t_s,x_s)}_{j,t_1}$ ( defined above) such that the restriction $f_{K^{(t_s,x_s)}_{j,t_1}}$ of $f$ to $K^{(t_s,x_s)}_{j,t_1}$ is smooth and has a singularity of order $(\lambda,\mu)$ for positive real numbers $\lambda,\mu$ at the point $(t_s,x_s)$.
Here, a smooth real-valued function $f_{K^{(t_s,x_s)}_{j,t_1}}\in C^{\infty}\left(  K^{(t_s,x_s)}_{j,t_1}\right) $ is said to have a singularity of order $(\lambda,\mu)$ for positive real numbers $\lambda,\mu$ at the point $(t_s,x_s)$, if for some $t_1\in U_j$ and cone  $K^{(t_s,x_s)}_{j,t_1}$
there is a finite constant $c>0$ such that for all $(t,x)\in K^{(t_s,x_s)}_{j,t_1}$ we have
\begin{equation}
{\big |}f(t,x){\big |}\leq \frac{c}{|t-t_s|^{\mu}|x-x_s|^{\lambda}},
\end{equation}
and if there are no $\lambda'<\lambda$ and $\mu'<\mu$ and a finite constant $c'$ such that
\begin{equation}
{\big |}f(t,x){\big |}\leq \frac{c'}{|t-t_s|^{\mu'}|x-x_s|^{\lambda'}}.
\end{equation} 
\end{defi}
The singularity analysis below shows that the velocity components $v_i$ themselves and their first order and second order spatial derivatives have no singularities from the left (of the type just described). Furthermore, we shall show that there is a uniform upper bound (from the left) at each time section ${t_s}\times {\mathbb T}^n$. This implies that a Leray-Hopf solution has a left-continuous extension. We shall combine this with local time regularity arguments in order to obtain global regularity results after any finite time for the Hopf-Leray solution.  
As a first consequence of the CKN-theory and singularity analysis we have
\begin{thm}\label{thmsing1}
For a given time horizon $T>0$ let $v_i,~1\leq i\leq n$ be a weak Leray-Hopf solution of the Navier Stokes equation, where  $v_i\in L^2([0,T],H^1)$ for all $1\leq i\leq n$. Then after any finite time this solution $v_i,¸1\leq i\leq n$  has no singularities of order $(\mu_0,\lambda_0)$ from the left for any $0<\lambda_0,\mu_0$. In other words, the Leray-Hopf solution is left-continuous an the time interval $(0,T]$ for arbitrary given $T$. Moreover, the assumption can be weakened assuming that $v_i\in L^2([0,T],H^{1-\epsilon})$ for any $\epsilon >0$, and after finite time there is still no singularity of order $(\mu_0,\lambda_0)$ from the left such that the relations in (\ref{para2}) are satisfied with $\mu_0>0$ or with $\lambda_0>0$.
\end{thm}
The theorem is proved in the next section. Next we consider the main idea of singularity analysis considered in this article, and state some further consequences. 
For each possible singular point $(t_s,x_s)$ of a given weak Hopf-Leray solution $v_i,~1\leq i\leq n$ there is an open cone $K^{(t_s,x_s)}_{j,t_1}$ where the Leray-Hopf solution component functions are smooth. We consider the coordinate transformation $c:K^{(t_s,x_s)}_{j,t_1}\rightarrow Z^{(t_s,x_s)}_{t_1}$, where
\begin{equation}
(t,x)\rightarrow (\tau,z)=\left(\frac{t}{t_s-t},\frac{x}{(t_s-t)^{\rho}} \right),~\rho\in (0,1.5). 
\end{equation}
Here $\rho$ is a parameter. For the analysis of this paper it is sufficient to assume that $\rho=1$ (such that $Z^{(t_s,x_s)}_{t_1}$ becomes a cylinder). This choice has the advantage that the transformed equation has neither (weakly) degenerate (for $\rho<1$) or singular (for $\rho >1$) second order coefficients. This implies  that except in the case $\rho=1$ the singularity analysis needs to be extended by an analysis of fundamental solution of heat equations with degenerate or weakly singular coefficients. Therefore, for simplicity, we shall stick to the case $\rho=1$ in this paper. 
Note that $Z^{(t_s,x_s)}_{t_1}$ is a cylinder of infinite height. Given a global weak Hopf-Leray solution and for a potential singularity at time $t_s$ we choose $t_1\in U_j=(t_j,t_s)$ for some $j\in J$ and 
 consider that global weak Hopf-Leray solution locally in transformed coordinates on the cylinder
\begin{equation}
 Z^{(t_s,x_s)}_{t_1}= \left[ t_{in},\infty\right)\times \Omega,
\end{equation}
where $\left[ t_{in},\infty\right):=\left[ \frac{t_1}{t_s-t_1},\infty\right)$.
For $1\leq i\leq n$ and all $(\tau,z)\in Z^{(t_s,x_s)}_{t_1}$ define 
\begin{equation}
w_i(\tau,z)=v_i(t,x).
\end{equation}
\begin{rem}
Note that we have a family of comparison functions $w_i=w_i^{(t_s,x_s)}$ here (for each possible singularity $(t_s,x_s)$ we construct one). If it is clear from the context that we refer to a given $(t_s,x_s)$ we drop this superscript for the sake of simplicity of notation.
\end{rem}

We have 
\begin{equation}
v_{i,j}=w_{i,j}\frac{dz_j}{dx_j}=w_{i,j}\frac{1}{(t_s-t)^{\rho}},~v_{i,j,j}=w_{i,j,j}\frac{1}{(t_s-t)^{2\rho}}.
\end{equation}
The function $w_i,~1\leq i\leq n$ is more regular with respect to time. Nevertheless, as $w_i(\tau,z)=v_i(t,x)$ we can transfer information obtained for $w_i,~1\leq i\leq n$ to $v_i,~1\leq i\leq n$. From the perspective of CKN theory this is an advantage: assuming $H^1$ regularity we have to push the regularity from $L^{\frac{8}{3}}\left(U_j,L^4\left({\mathbb T}^3\right)   \right)$ to $L^{8}\left(U_j,L^4\left({\mathbb T}^3\right)   \right)$ in order to obtain uniqueness, and if we can do this for $w_i,¸1\leq i\leq n$, then we can do it for the velocity components $v_i,~1\leq i\leq n$ of the correspeonding Leray-Hopf solution itself. Moreover, if we can weaken the assumption of $H^1$-regularity to $H^{1-\epsilon}$-regularity in this context, then we have uniqueness after any finite time in the situation of Hopf's theorem. More precisely, if for all possible singularities $(t_s,x_s)$ with $t_s >0$ and for a cylinder $Z^{(t_s,x_s)}_{j,t_1}=[t_{in},t_s]\times \Omega$ we have
$$w_i\in L^{\infty}\left([t_{in},t_s],L^2\left(\Omega\right)  \right)\cap L^2\left( [t_1,t_s],H^1\left(\Omega\right) \right)\cap  L^{8}\left(\left[t_{in},t_s\right],L^4\left(\Omega\right)   \right),$$
then after small time $t_0 >0$ we have 
$$v_i\in L^{\infty}\left([t_0,T],L^2\left({\mathbb T}^n\right)  \right)\cap L^2\left( [0,T],H^1\left({\mathbb T}^n\right) \right)\cap  L^{8}\left(\left[0,T\right],L^4\left({\mathbb T}^n\right)   \right).$$
Note that for strong data, say for data in $H^m\cap C^m,~m\geq 2$ the latter statement implies uniqueness and regularity for all time because we have a local time contraction result for such strong function spaces at initial time $t=0$.
Therefore, even from the the perspective of weak function spaces the study of the more regular function $w_i, 1\leq i\leq n$ has some advantages.  For this reason we analyze the behavior of this function on a cylinder which corresponds to a cone in original coordinates.  
Next, we derive initial-boundary value problems for each cone with an assumed singularity at the tip of a related cone. First we observe that the incompressibility condition is conserved on the time interval $[t_{in},t_s)$, since
\begin{equation}
0=\div \mathbf{v}=\sum_{i=1}^nv_{i,i}=\sum_{i=1}^nw_{i,i}\frac{1}{(t_s-t)^{\rho}}.
\end{equation}
For the time derivative we have (recall that $\rho=1$)
\begin{equation}
v_{i,t}=w_{i,\tau}\frac{d\tau}{dt}+\sum_{j=1}^nw_{i,j}\frac{x_j}{(t_s-t)^{2}},
\end{equation}
where for $t\in [0,t_s)$ we have
\begin{equation}
\begin{array}{ll}
\frac{d\tau}{dt}=\frac{d }{dt}\frac{t}{(t_s-t)}=\frac{1}{(t_s-t)}+\frac{t}{(t_s-t)^2}
=\frac{t_s-t+t}{(t_s-t)^2}=\frac{t_s}{(t_s-t)^2}>0.
\end{array}
\end{equation}
Hence the function $w_i,~1\leq i\leq n$ is determined by initial-boundary value problem
\begin{equation}\label{Navtrans}
\left\lbrace \begin{array}{ll}
\frac{\partial w_i}{\partial \tau}-\mu_2\nu \sum_{j=1}^n w_{i,j,j}
+\mu_1\sum_{j=1}^n\left(  w_j+z_j\right) w_{i,j}=-\mu_1\nabla_i p^w,\\
\\
\mbox{div}~\mathbf{w}=0,\\
\\
w_i|_{\partial^S Z^{(t_s,x_s)}_{t_1}}=v_i|_{\partial^sK^{(t_s,x_s)}_{j,t_1}},\\
\\
\mathbf{w}\left( \frac{t_1}{t_s-t_1},.\right) =\mathbf{v}(t_1,.),
\end{array}\right.
\end{equation} 
where $\partial^S Z^{(t_s,x_s)}_{t_1}$ denotes the spatial boundary of the cylinder  $Z^{(t_s,x_s)}_{t_1}$, and
\begin{equation}
\begin{array}{ll}
\mu_2\equiv \mu_2(\tau)=\frac{1}{t_s}(t_s-t)^{2-2\rho},~\mu_1\equiv \mu_1(\tau)=\frac{1}{t_s}(t_s-t)^{2-\rho},\\
\end{array}
\end{equation}
and where $t\equiv t(\tau)$ denotes the value at $\tau$ of the inverse $\tau\rightarrow t(\tau)$ of $\tau\equiv \tau(t)$. Note that we used
\begin{equation}
z_j=(t_s-t)^{\rho}x_j.
\end{equation}
Furthermore, $p^w$ denotes the pressure in transformed coordinates. Note that $\mu_2$ is bounded for our choice $\rho=1$ in this article and becomes weakly singular for $\rho\in (1,1.5)$. 
\begin{rem}
We shall later show that we have a regularity transfer from $w-i,~1\leq i\leq n$ to $v_i,~1\leq i\leq n$, i.e., full regularity of the comparison functions $w_i,~1\leq i\leq n$ (for each cone one) implies full regularity of the original velocity function $v_i,~1\leq i\leq n$. Note that for the sake of notational simplicity we drop the reference to the cone in general, i.e., we have $w_i\equiv w^{Z^{(t_s,x_s)}_{t_1}}_i=v^{K^{(t_s,x_s)}_{t_1}}_i$ for each cone $K^{(t_s,x_s)}_{t_1}$, and where for all $1\leq i\leq n$ $v^{K^{(t_s,x_s)}_{t_1}}_i$ denotes the restriction of $v_i$ to the cone $K^{(t_s,x_s)}_{t_1}$.
For refined regularity investigations we may use a variation of initial-Neumann-boundary value problems of the form
\begin{equation}\label{Navtrans2}
\left\lbrace \begin{array}{ll}
\frac{\partial w_i}{\partial \tau}-\mu_2\nu \sum_{j=1}^n w_{i,j,j}
+\mu_1\sum_{j=1}^n\left(  w_j+z_j\right) w_{i,j}=-\mu_1\nabla_i p^w,\\
\\
\mbox{div}~\mathbf{w}=0,\\
\\
\partial_{\nu}w_i|_{\partial^S Z^{(t_s,x_s)}_{t_1}}=(t_s-t)^{\rho}\partial_{\nu}v_i|_{\partial^sK^{(t_s,x_s)}_{j,t_1}},\\
\\
\mathbf{w}\left( \frac{t_1}{t_s-t_1},.\right) =\mathbf{v}(t_1,.),
\end{array}\right.
\end{equation}
where $\partial_{\nu}$ denotes the normal spatial derivative. Indeed in case $\rho=1$ we gain regularity of one order for the estimation of the boundary terms in the estimation of $w_i$ if we use Neumann boundary condition (factor $(t_s-t)^{\rho}$).
\end{rem}

\begin{rem}
The term 'weakly singular' means -roughly- 'integrable'. Furthermore, note that $\lim_{\tau\uparrow \infty}\mu_1(\tau)=0$, and $\lim_{\tau\uparrow \infty}\mu_0(\tau)=0$, and these coefficients are bounded in any case. We shall analyze the role of  the coefficients at time $t=t_s$ in case $\rho\in (1,1.5)$ elsewhere, where we reconsider the Levy expansion of the fundamental solution in this context.
\end{rem}
The problem described in (\ref{Navtrans}) is a initial-boundary value problem.
The initial time may be abbreviated by
\begin{equation}\label{ti}
t_{in}:=\frac{t_1}{t_s-t_1},~\mbox{where}~ t_s>0.
\end{equation}
In the case $\rho=1$ and $\mu_2=\frac{1}{t_s}$ is a constant. Hence, the fundamental solution $G^{\mu}_{\nu}$ of the transformed heat equation
\begin{equation}
q_{,\tau}-\mu_2\nu \Delta q=0
\end{equation}
is an explicitly known Gaussian type function (no expansion is needed).  Applying the divergence operator to the first equation in (\ref{Navtrans}) we get
\begin{equation}
\begin{array}{ll}
\sum_{i}\frac{\partial w_{i,i}}{\partial \tau}-\mu_2\nu \sum_{j=1}^n\sum_{i} w_{i,i,j,j}
+\mu_1\sum_i\sum_{j=1}^n w_{j,i} w_{i,j}+\mu_1\sum_{i,j=1}^n \delta_{ji} w_{i,j}\\
\\
+\mu_1\sum_{j=1}^n \left(w_{j}+z_j\right) \sum_iw_{i,i,j}=\mu_1\sum_i\sum_{j=1}^n (w_{j,i}+\delta_{ji}) w_{i,j}=
-\mu_1\Delta p^w,
\end{array}
\end{equation}
as the incompressibility condition transfers to the function $w_i, 1\leq i\leq n$, and where we recall that we add the superscript $w$ to the pressure in order to indicate that the pressure is considered in transformed coordinates.
 The Leray projection form of (\ref{Navtrans}) is determined via
\begin{equation}\label{laplace}
\Delta p^w=-\sum_{i,j=1}^nw_{i,j}w_{j,i}.
\end{equation}
Hence the Leray projection form is obtained from (\ref{Navtrans}) if we replace the first dynamical equation by
\begin{equation}
\frac{\partial w_i}{\partial \tau}-\mu_2\nu \sum_{j=1}^n w_{i,j,j}
+\sum_{j=1}^n\mu_1\left( w_j+z_j\right) w_{i,j}=\mu_1L^i_{{\mathbb T}^n}\left(\sum_{j,k=1}^nw_{j,k}w_{k,j}\right).
\end{equation}
Here, $L^i_{{\mathbb T}^n}$ denotes the Leray projection operator in case of the $n$-torus (corresponding to the $i$th derivative of the pressure).
This operator can be determined explicitly by Fourier transformation, which we shall do below in the singularity analysis. 

Hence, in the domain  $Z^{(t_s,x_s)}_{t_1}$ (resp. the cutoff domain  $Z^{(t_s,x_s)}_{t_1,\tau}:=[t_{in},\tau]\times \Omega$) we have a the classical representation
\begin{equation}\label{wrep}
\begin{array}{ll}
 w_i(\tau,z)=\int_{ y\in \Omega }w_i(t_{in},y)G^{\mu}_{\nu}(\tau-t_{in},z-y)dy\\
\\
-\int_{Z^{(t_s,x_s)}_{t_1,\tau}}\left( \sum_{j=1}^n \mu_1 ( w_j+z_j)w_{i,j}\right)(s,y)G^{\mu}_{\nu}(\tau-s,z-y)dyds\\
\\ +\int_{Z^{(t_s,x_s)}_{t_1,\tau}}\mu_1(s)L^i_{{\mathbb T}^n} 
\left( \sum_{j,m=1}^n\left( w_{m,j} w_{j,m} \right) \right)(s,y)G^{\mu}_{\nu}(\tau-s,z-y)dyds\\
\\
+\int_{t_1}^{\tau}\int_{\partial^S Z^{(t_s,x_s)}_{t_1,\tau}}w^{\partial^Z}_i(s,y)G^\mu_{\nu}(\tau,z;s,y)dyds.
\end{array}
\end{equation}
Here, for the boundary term we use (\ref{wvbd}) below. In the following the abbreviation $N^i$ refers to   the sum of the Burgers term and the Leray projection term (cf. (\ref{Ni}) below).
 Then the boundary  terms are determined by 
\begin{equation}
\begin{array}{ll}
w^{\partial^Z}_i(\tau,z)=-2v_i{\big |}_{S_K}(\tau,x(z))+2\int_{\Omega}w_i(t_{in},y)G^{\mu}_{\nu}(\tau,z;t_{in},y)dy+\left( 2N^i\ast G^{\mu}_{\nu}\right) (\tau,z)\\
\\
\sum_{k=1}^{\infty}\int_{t_{in}}^{\tau}\int_{\Omega}{\Big (} -2v_i{\big |}_{S_K}(\sigma,x(\xi))+2\int_{\Omega}w_i(t_{in},y)G^{\mu}_{\nu}(\sigma,\xi;t_{in},y)dy\\
\\
-\left( 2N^i\ast G^{\mu}_{\nu}\right) (\sigma,\xi){\Big )}\times G^{\mu,k}_{\nu}(\tau,z,\sigma,\xi)d\xi d\sigma,
\end{array}
\end{equation}
where $t_{in}$ is defined in (\ref{ti}) above, $S_K$ denotes the spatial boundary of the cone, $z\rightarrow x(z)$ denotes the transformation from the spatial boundary of the cylinder $\partial^SZ^{(t_s,x_s)}$ to the spatial boundary of the cone $K^{(t_s,x_s)}_{j,t_1}$, and the series $G^{\mu,k}_{\nu},~k\geq 2$ is defined recursively by
\begin{equation}
\begin{array}{ll}
G^{\mu,1}_{\nu}:=G^{\mu}_{\nu},\\
\\
~G^{\mu,k+1}_{\nu}(\tau,u;s,v):=\int_{t_{in}}^{\tau}\int_{\Omega}G^{\mu}_{\nu}(\tau,u;\sigma,w)G^{\mu}_{\nu}(\sigma,w;s,v)dS_yd\sigma .
\end{array}
\end{equation}
Furthermore, note that 
\begin{equation}\label{wvbd}
w_i(\tau,z){\big |}_{\partial^S Z^{(t_s,x_s)}_{t_1}}=v_i{\big |}_{S_K}(t,x),
\end{equation}
and
\begin{equation}\label{Ni}
\begin{array}{ll}
\left( 2N^i\ast G^{\mu}_{\nu}\right) (\tau,z)=\\
\\
-2\int_{Z^{(t_s,x_s)}_{t_1}}\left( \sum_{j=1}^n \mu_1 ( w_j+z_j)w_{i,j}\right)(s,y)G^{\mu}_{\nu}(\tau-s,z-y)dyds\\
\\ 
+2\int_{Z^{(t_s,x_s)}_{t_1}}\mu_1(s)L^i_{{\mathbb T}^n} 
\left( \sum_{j,m=1}^n\left( w_{m,j} w_{j,m} \right) \right)(s,y)G^{\mu}_{\nu}(\tau-s,z-y)dyds
\end{array}
\end{equation}

Here, recall that $G^{\mu}_{\nu}$ is the fundamental solution of the equation $G^{\mu}_{\nu,\tau}-\mu_2\nu \Delta G^{\mu}_{\nu}=0$. Furthermore, the coefficient $\mu_2$ is bounded for  $\rho=1$, but for $\rho\in (0,1)$ or $\rho\in (1,1.5)$ we have a 'degenerate' or 'weakly singular' coefficients $\mu_2$ respectively. 


Using such classical representations of the transformed comparison function $w_i,~1\leq i\leq n$, Theorem \ref{thmsing1} is proved in section 2. 
For each argument $(t_s,x_s)$ 
of a possible singularity of a velocity component functions 
 we find an upper bound and 
we can show that for each time $t_s$ we can find an upper bound for a family of comparison functions which do not depend on $x_s$.  As a consequence the superior limites of the velocity component functions have a lower and upper bound although we cannot construct this upper bound by the methods in \cite{CKN}. On this abstract level we can argue: if there is no upper bound of the modulus of a velocity component function, then we can find a point in the Hausdorff set with a $(\lambda,\mu)$-singularity which satisfies $0\leq \mu_0 <\frac{1}{2},~0\leq \lambda_0 < \frac{3}{2}$.. However Theorem \ref{thmsing1} tells us that this is not possible. We have even more. The CKN-theory tells us that for the set $S$ of singularities we have
\begin{equation}
H^1(S)=\lim_{\delta \downarrow 0}H^{1,\delta}(S)=0
\end{equation}
where for $1\leq p\leq n$
\begin{equation}
H^{p,\delta}=\left\lbrace \sum_{j=1}^{\infty}\left(\mbox{diam}B_j\right)^p:S\subset \cup_{j=1}^{\infty}B_j \mbox{ and } \mbox{diam}(B_j)\leq \delta\right\rbrace .
\end{equation}
Let $S_{t_s}:=\left\lbrace (t,x)\in S|t=t_s\right\rbrace$. We have $S_{t_s}\subset S$ and
\begin{equation}
H^n(S_{t_s})\leq H^1(S_{t_s})\leq H^1(S)=0
\end{equation}
and it is well-known that for $\gamma_n=\frac{\pi^{n/2}}{2^n}\Gamma\left(n/2+1\right)$
\begin{equation}
\gamma_nH^n \mbox{ is Lebegues measure.}
\end{equation}
Hence sharpening Theorem \ref{thmsing1} for derivatives of the functions $w_i,~1\leq i\leq n$ with uniform upper bounds and transfer of the result to the original velocity components $v_i,¸1\leq i\leq n$ leads to the possibility of regular extension of these functions to the time section $\left\lbrace t_s\right\rbrace \times {\mathbb T}^n$ for all finite $t_s >0$.
Note that for multivariate spatial derivatives of finite order $\alpha=(\alpha_1,\cdots,\alpha_n)$ we have  for $|\alpha|\geq 1$
\begin{equation}\label{deralpha}
{\big |}D^{\alpha}_xv_{i}(t,.){\big |}\leq c_{\alpha}{\big |}D^{\alpha}_{z}w_{i}(\tau,.){\big |}\frac{1}{(t_s-t)^{\rho|\alpha|}},
\end{equation}
where $|.|$ denotes the spatial supremum norm, i.e., $|f|:=\sup_{x\in {\mathbb R}^n}|f(x)|$ for a function $f:{\mathbb R}^n\rightarrow {\mathbb R}$, and $|\alpha|:=\sum_{i=1}^n\alpha_i$. Hence, if we estimate multivariate spatial derivatives of the velocity components of a Leray-Hopf solution  using local solution representations as in (\ref{wrep}) it becomes increasingly diffcult to get a regular upper bound as we aim at upper bounds for higher order derivatives. Nevertheless we can transfer regularity results for the comparison function $w_i,¸1\leq i\leq n$ to the corresponding original Hopf-Leray solution using spatial Lipschitz continuity of the Euler-Leray data function applied to regular data (which is the application of the Leray projection term operator to given data). We may interpret the initial value problem for $v_i,¸1\leq i\leq n$ as an initialboundary value problem, where we consider the data on the boundary of a cylinder to be given by the solution. We then have a similar representation as in (\ref{wrep}) above. More precisely, for some constants $0\leq c_1<c_2\leq 1$ and in an interval $[t_1,t_s)$ the original velocity function $v_i,¸1\leq i\leq n$ has a local time representation on a cylinder $Z^v_{t_1,c_1,c_2}:=\left\lbrace (t,x)|t\in [t_1,t_s)~\&~x\in [c_1,c_2]\right\rbrace$ of the form  

\begin{equation}\label{vrep}
\begin{array}{ll}
 v_i(t,x)=\int_{\left\lbrace y|(t_1,y)\in Z^{v}_{t_1,c_1,c_2}\right\rbrace }v_i(t_1,y)G_{\nu}(t,x-y)dy\\
\\
-\int_{Z^{v}_{t_1,c_1,c_2}}\left( \sum_{j=1}^n  v_jv_{i,j}\right)(s,y)G_{\nu}(t-s,x-y)dyds\\
\\ +\int_{Z^{v}_{t_1,c_1,c_2}}L^i_{{\mathbb T}^n} 
\left( \sum_{j,m=1}^n\left( v_{m,j} v_{j,m} \right) \right)(s,y)G_{\nu}(t-s,x-y)dyds\\
\\
+\int_{t_1}^{t}\int_{\partial^S Z^{v}_{t_1,c_1,c_2}}v^{\partial^Z}_i(s,y)G_{\nu}(t,x;s,y)dyds,
\end{array}
\end{equation}
where $G_{\nu}$ is $G^{\mu}_{\nu}$ in case $\mu=1$, and where the boundary term has an analogous definition as in the case of the comparison function in (\ref{wrep}) above.
Assume that full regularity and the existence of an uniform upper bound is proved for a family of comparison functions $w^{(t_s,x_s)}_i,~1\leq i\leq n$ for all $\tau=\tau(t)$ for  $(t_s,x_s)\in S_{t_s}:=\left\lbrace (t,x)\in S|t=t_s ,\right\rbrace$ (a section at $t_s$ of the CKN-Haussdorff set of possible singularities), such that for some $m\geq 2$
\begin{equation}\label{condw}
\begin{array}{ll}
\forall (t_s,x_s)\in S_{t_s}:~w^{t_s,x_s}_i\in C\left([\tau(t_1),\tau(t_s)],H^m\cap C^{m}\right),\\
\\
~\sup_{(t_s,x_s)\in S_{t_s}}\sup_{\tau >t_1}|w^{(t_s,x_s)}_i(\tau,.)|_{H^m}\leq C<\infty. 
\end{array}
\end{equation}
In the following we drop reference to a specific singularity at $(t_s,x_s)$ for simplicity of notation, i.e., we write
\begin{equation}
w_i=w_i^{t_s,x_s}, \mbox{ if the reference to $(t_s,x_s)\in S$ is known form the context.}
\end{equation}
Here, by an uniform regular upper bound we mean that $C$ is independent of the singular point at time $t_s$, i.e., independent of $(t_s,x_s)\in S_{t_s}$.
For the regularity transfer the Leray projection term is crucial. We have
\begin{equation}
p_{,i}=p^w_{,i}\frac{dz_i}{dx_i}=p^w_{,i}\frac{1}{(t_s-t)}
\end{equation}
Hence, regularity of $w_i,~1\leq i\leq n$ and (\ref{condw}) implies that 
\begin{equation}
{\Bigg |}L^i_{{\mathbb T}^n} 
\left( \sum_{j,m=1}^n\left( v_{m,j} v_{j,m} \right) \right)(s,x){\Bigg |}\leq \frac{C}{(t_s-s)^{\rho}}.
\end{equation}
Recall that we consider $\rho=1$ such that we have the Gaussian $G_{\nu}$ in (\ref{vrep}). A main  idea for the regularity transfer is to use a spatial symmetry of first order spatial derivatives of the Gaussian in convolutions with the Leray projection term, where we use this symmetry along with (local) spatial Lipschitz continuity of the Leray projection term, i.e., the relation
\begin{equation}
{\Big |}L^i_{{\mathbb T}^n} 
\left( \sum_{j,m=1}^n\left( v_{m,j} v_{j,m} \right) \right)(s,x-y)-L^i_{{\mathbb T}^n} 
\left( \sum_{j,m=1}^n\left( v_{m,j} v_{j,m} \right) \right)(s,x-y'){\Big |}\leq l|y-y|'
\end{equation} 
for a Lipschitz constant which is independent of $x$ (inherited from spatial Lipschitz continuity of the Leray projection term $L^i_{{\mathbb T}^n} \left( \sum_{j,m=1}^n\left( w_{m,j} w_{j,m} \right) \right)(s,x),~1\leq i\leq n$ proved in detail below).
\begin{rem}
Note that local spatial Lipschitz continuity of the Leray projection term is needed for our purposes, because we estimate convolutions of the Gaussian and (mainly of first order)  spatial derivatives of the Gaussian. 
\end{rem}
%
 We shall prove below, that a refinement of this idea leads to the conclusion of a full transfer of regularity from the comparison function $w_i,~1\leq i\leq n$ to the original velocity function $v_i,~1\leq i\leq n$. 
We start with the estimate of the velocity components themselves and consider the proof of theorem \ref{thmsing1} first. Furthermore we first construct regular extensions  of the function $w_i,¸1\leq i\leq n$ to the CKN-Hausdorff set of possible singularities. We then transfer the results to the original velocity component functions.

%
We get
\begin{thm}\label{thmsing2}
Let $T>0$ be a given horizon, and let $v_i,~1\leq i\leq n$ be a weak Leray-Hopf solution of the Navier Stokes equation, where  $v_i\in L^2([0,T],H^1)$for all $1\leq i\leq n$. Then after any finite time the velocity component functions $v_i,¸1\leq i\leq n$ and their  spatial derivatives up to second order can be continuously extended to the Haussdorff set of possible singularities predicted by CKN-theory. 
The result can be sharpened such that it also holds if a weak Leray-Hopf solution $v_i,~1\leq i\leq n$ satisfies $v_i\in L^2([0,T],H^{1-\epsilon})$ for all $t\geq 0$ and $\epsilon >0$ small.
\end{thm}
The result in (\ref{thmsing2}) does not tell us how the  uniform upper bounds depend on $t_s$ as $t_s\downarrow 0$. It seems that the analysis of this paper can be extended in order to show that $H^1$-data (maybe even $H^{1-\epsilon}$-data for small $\epsilon >0$) are sufficient for global existence, uniqueness and smoothness. In any case, local time existence, which is available for strong data implies the existence and regularity for all time.  It is therefore quite possible that a stronger conclusion (with with $H^1$ or $H^{1-\epsilon}$ instead of $H^2\cap C^2$ data or  $H^{\frac{n}{2}+1}$-data) of CKN-theory than the following can be obtained. Nevertheless, we note
\begin{cor}
For $v_i\in H^1$ for all $1\leq i\leq n$ with respect to time and space global existence and uniqueness holds for all time. Moreover if a local time contraction results for the initial data, then global existence and uniqueness hold for all time $t\in [0,\infty)$. This is true for $v_i(0,.)\in H^m\cap C^m,¸m\geq 2$ on the whole space and for $v_i(0,.)\in H^{2.5}$ for the torus in dimension $n=3$ (sufficient criteria, cf. also appendix).
\end{cor}
The local contraction results were proved elsewhere. These local contraction results are constructive, and they may be of independent value for numerical analysis. The CKN-theory is an example of the additional power of non-constructive analysis. Although non-constructive methods cannot be justified from the constructive point of view, even from this constructive point of view non-constructive methods make predictions which are then proved later by constructive methods with much more effort (cf. Hilbert's early paper on polynomial invariants and K\"{o}nig's  later constructive argument).        
It is remarkable that the statements of Theorem \ref{thmsing1} and of Theorem \ref{thmsing2} are true under the weaker condition that a weak Leray-Hopf solution  which satisfies $v_i\in L^2([0,T],H^{1-\epsilon})$ for any small $\epsilon >0$. This shows the full power of the CKN-theory combined with the singularity analysis, which we consider next in the following proofs.

\section{Proof of theorem \ref{thmsing1}}
For simplicity we consider $\rho=1$ such that $\mu_2$ is a constant and we have $G_{\nu'}=G^{\mu}_{\nu}$ with adjusted viscosity $\nu'=\nu\mu_2$.
As a consequence of (\ref{wrep}) we have
\begin{equation}\label{wrep2}
\begin{array}{ll}
{\big |}v_i(t_s,x_s){\big |}\leq \sup_{\tau\uparrow \infty,z^{\tau}_s:=\frac{x_s}{t_s-t(\tau)}} {\big |}w_i\left( \tau,z^{\tau}_s\right) {\big |}\\
\\
 \leq \sup_{\tau\uparrow \infty,z^{\tau}_s}{\Big (}\int_{\left\lbrace y|(t_1,y)\in Z^{(t_s,x_s)}_{t_1}\right\rbrace }w_i(t_{in},y)G_{\nu'}(\tau,z^{\tau}_s-y)dy{\big |}\\
\\
+{\big |}\int_{Z^{(t_s,x_s)}_{t_1}}\left( \sum_{j=1}^n \mu_1( w_j+z_j)w_{i,j}\right)(s,y)G_{\nu'}(\tau-s,z^{\tau}_s-y)dyds{\big |}\\
\\ 
+{\big |}\int_{Z^{(t_s,x_s)}_{t_1}}\mu_1(s)L^i_{{\mathbb T}^n} 
\left( \sum_{j,m=1}^n\left( w_{m,j} w_{j,m} \right) \right)(s,y)G_{\nu'}(\tau-s,z^{\tau}_s-y)dyds{\big |}\\
\\
+{\big |}\int_0^{\tau}\int_{\partial^S Z^{(t_s,x_s)}_{t_1}}w^{\partial^Z}_i(s,y)G_{\nu'}(\tau,z^{\tau}_s;s,y)dyds{\big |}{\Big )}.
\end{array}
\end{equation}
Shifting spatial coordinates we may assume that $x_s=0$ if this is convenient.
Next we estimate the terms on the right side of (\ref{wrep2}). The Leray projection term is crucial, and we start with this term first. In a first step we consider the Leray projection operator more closely. As remarked above it is determined by the Poisson equation in (\ref{laplace}) which is the same as the corresponding equation for the original velocity component $v_i$ by a property of the transformation.  
We prove
\begin{lem}\label{lplem}  Assume that $n\leq3$ and that for all $t\in [t_1,t_s]$ a Leray-Hopf solution satisfies $v_i(t,.)\in H^1$. Then
\begin{equation}
p_{,i}\in L^2.
\end{equation}
 The result still holds under the weaker assumption that for all $t\in [t_1,t_s]$ we have $v_i(t,.)\in H^{1-\epsilon_0}$ for small $\epsilon_0>0$. Here we note that the upper bound constants used are global, i.e., within a fixed arbitrary time horizon $T>0$ they do not depend on some argument $(t_s,x_s)$ of a possible singularity.
\end{lem}\label{pressure}
\begin{proof}
For a Hopf-Leray solution $v_{i}\in L^2\left( [t_1,t_s], H^1\right) $ we have
\begin{equation}\label{sing+}
\mbox{ for all }~(t,x)\in \overline{K^{(t_s,x_s)}_j}:~|v_{i,k}(t,x)|\leq \frac{c}{(t_s-t)^{\delta_0}|x-x_s|^{\lambda_0}}
\end{equation}
for some finite constant $c\in {\mathbb R}_+$ and for some parameters $\delta,\lambda$, which satisfy (case $n=3$) the  relation
\begin{equation}\label{para*}
0\leq \delta_0 <\frac{1}{2},~0\leq \lambda_0 < \frac{3}{2}.
\end{equation}
The function $v_i,~1\leq i\leq n$ is defined on the whole torus, such that from the  representation
\begin{equation}
v_i(t,x):=\sum_{\alpha\in {\mathbb Z}^n}v_{i\alpha}\exp\left( 2\pi i\alpha x\right),
\end{equation}
with time-dependent modes $v_{i\alpha}$ we get
\begin{equation}\label{navode200first*}
\begin{array}{ll}
p_{,i}(t,x)=:\sum_{\alpha\in {\mathbb Z}^n}p_{\alpha i}\exp\left( 2\pi i\alpha x\right)
\\
\\
=\sum_{\alpha\in {\mathbb Z}^n}2\pi i\alpha_i1_{\left\lbrace \alpha\neq 0\right\rbrace}\frac{\sum_{j,k=1}^n\sum_{\gamma\in {\mathbb Z}^n}4\pi^2 \gamma_j(\alpha_k-\gamma_k)v_{j\gamma}v_{k(\alpha-\gamma)}}{\sum_{i=1}^n4\pi^2\alpha_i^2}\exp\left( 2\pi i\alpha x\right),
\end{array} 
\end{equation}
where $p_{,i\alpha}$ denotes the $\alpha$-mode  of $p_{,i}$.
Note that the infinite vector of time-dependent modes $\mathbf{v}^F_i=(v_{i\alpha}(t))_{\alpha\in {\mathbb Z}^n}$ of the velocity component $v_i(t,.)$ is in the dual Sobolev space of order $s\in {\mathbb R}$ iff
 \begin{equation}
 \sum_{\alpha \in{\mathbb Z}^n}|v_{i\alpha}|^2\left\langle \alpha\right\rangle^{2s}< \infty, 
 \end{equation}
where
\begin{equation}
 \left\langle \alpha\right\rangle :=\left(1+|\alpha|^2 \right)^{1/2}. 
\end{equation}
Since $v_i(t,.)\in H^1,¸1\leq i\leq n$, we have 
\begin{equation}
\sum_{\beta\in {\mathbb Z}^n}|\beta|^2(v_{i\beta}(t))^2<\infty.
\end{equation}
Hence,
\begin{equation}
\sum_{\gamma\in {\mathbb Z}^n}4\pi^2 \gamma_j(\alpha_k-\gamma_k)v_{j\gamma}v_{k(\alpha-\gamma)}\leq \sup_{t\in [t_1,t_s]}c_*(t(\tau))=: c<\infty,
\end{equation}
where we may choose (use $ab\leq \frac{1}{2}\left( a^2+b^2\right)$) such that
\begin{equation}
c^*(t(\tau)):=\sum_{\gamma\in {\mathbb Z}^n}v^2_{j\gamma}(\tau).
\end{equation}
 It follows that
\begin{equation}\label{navode200first}
\begin{array}{ll}
|p_{,i\alpha}(t)|\leq \sum_{\alpha\in {\mathbb Z}^n}|2\pi i\alpha_i|1_{\left\lbrace \alpha\neq 0\right\rbrace}\frac{n^2c}{\sum_{i=1}^n4\pi^2\alpha_i^2},
\end{array} 
\end{equation}
where $p_{,i\alpha}$ denotes the $\alpha$-mode of $p_{,i}$. For given $t\in [t_1,t_s]$ the square $|p_{,i\alpha}(t)|^2$ has an integrable upper bound, where we note that the factor $\alpha_i$ in the numerator occurs only on one dimension. Hence for $n\leq 3$ we have $p_{,i}(t,.)\in L^2$. Hence, since $(t_s,x_s)$ is the only singular point of $p_{,i}$ on the cone $K^{(t_s,x_s)}_j$ we have the spatial upper bound
\begin{equation}
\mbox{ for all $(t,x)\in K^{(t_s,x_s)}_j$}~|p_{,i}(t,x)|\leq \frac{c}{(t_s-t)^{\delta_1}|x_s-x|^{3/2-\epsilon}},
\end{equation}
where $\delta_1\in \left(0,1\right)$ can be close to $1$ because $\delta_0\in  \left(0,0.5\right)$. Finally, if for all $t\in [t_1,t_s]$ we have $v_i(t,.)\in H^{1-\epsilon_0}$ for small $\epsilon_0>0$, then 
\begin{equation}\label{navode200first*}
\begin{array}{ll}
|p_{,i\alpha}(t)|\leq 
\sum_{\alpha\in {\mathbb Z}^n}|2\pi i\alpha_i|1_{\left\lbrace \alpha\neq 0\right\rbrace}
\frac{n^2c}{\sum_{i=1}^n4\pi^2\alpha_i^{2(1-\epsilon')}},
\end{array} 
\end{equation}
for small $\epsilon'>0$ such that we still have $p_{,i}(t,.)\in L^2$ for $n\leq 3$ and the result still holds.
Here we note again that the factor $\sum_{\alpha_i\in {\mathbb Z}}{\Big |}\frac{\alpha_i}{\alpha_i^2}{\Big |^2}=\sum_{\alpha_i\in {\mathbb Z}}{\Big |}\frac{1}{\alpha_i^2}{\Big |}$ in (\ref{navode200first*}) is finite and the  remaining $n-1$ dimensional sum is also finite for $n\leq 3$ (comparison with integral upper bounds).
 
\end{proof}

Next, we have the pointwise relation
\begin{equation}
p_{,i}=p^w_{,i}\frac{dz_i}{dx_i}=p^w_{,i}\frac{1}{(t_s-t)^{\rho}}.
\end{equation}
Note that for $\rho=1$ we have
\begin{equation}
\tau=\frac{t}{(t_s-t)}\rightarrow (t_s-t)\tau =t\rightarrow t_s\tau =t(1+\tau)
\rightarrow t=\frac{t_s\tau}{1+\tau}, 
\end{equation}
from which
\begin{equation}
t_s-t=\frac{t_s(1+\tau)}{1+\tau}-\frac{t_s\tau}{1+\tau}=\frac{t_s}{1+\tau}
\end{equation}
follows. Hence, the function
 \begin{equation}
{\mathbb R}^+\times {\mathbb R}^n\ni(\tau,z)\rightarrow {\big |}p^w_{,i}(\tau,z){\big |}(t_s-t)^{-1}={\big |}p^w_{,i}(\tau,z){\big |}\frac{1+\tau}{t_s}\in L^2,
 \end{equation}
and, since $p^w$ is smooth, we have for some $\delta_2\in (0,0.5)$ 
\begin{equation}
{\big |}p^w_{,i}(\tau,z){\big |}\leq \frac{c}{(1+\tau)^{1+\delta_2}(1+|z|^{1.5-\epsilon})},
\end{equation}
where $p^w_{,i}(\tau,z)=L^i_{{\mathbb T}^n} 
\left( \sum_{j,m=1}^n\left( w_{m,j} w_{j,m} \right) \right)(\tau,z)$.
Next recall that for $\delta \in (0,1)$ we have the standard estimate
\begin{equation}
{\big |}G_{\nu'}(\sigma,y){\big |}\leq \frac{c}{|\sigma|^{\delta}|y|^{n-2\delta}},
\end{equation}
or, alternatively, since the integral with respect time is for $\sigma \geq t_{in}>0$ we can even use the obvious estimate
\begin{equation}
{\big |}G_{\nu'}(\sigma,y){\big |}\lesssim \frac{1}{\sqrt{\sigma}^3}.
\end{equation}
Anyway, recall that Leray projection term is the convolution of the latter two terms times $\mu_1=(t_s-t)/t_s$. Choosing $x_s=0=z^{\tau}_s$ for $\tau\in (0,\infty)$ w.l.o.g. we get 
\begin{equation}
\begin{array}{ll}
\sup_{\tau\uparrow \infty,z^{\tau}_s}{\big |}\int_{t_{in}}^{\tau}\int_{|y|\leq 1}\mu_1(t(\sigma))L^i_{{\mathbb T}^n} 
\left( \sum_{j,m=1}^n\left( w_{m,j} w_{j,m} \right) \right)(\sigma,y)\times\\
\\
\times G_{\nu'}(\tau-\sigma,z^{\tau}_s-y)dyd\sigma{\big |}\\
\\
\leq \sup_{\tau\uparrow \infty,z^{\tau}_s}{\big |}\int_{t_{in}}^{\tau}\int_{|y|\leq 1}
\frac{c}{(1+\sigma)^{2+\delta_1}
(|y|^{1.5-\epsilon})}\frac{c}{|\tau-\sigma|^{\delta}|z^{\tau}_s-y|^{3-2\delta}}dyd\sigma{\big |}\\
\\
\leq \tilde{c}+ \sup_{\tau\uparrow \infty}{\big |}\int_{t_{in}}^{\tau}\int_{|y|\leq 1}
\frac{c}{(1+\tau)^{1+\delta_1+\delta}
(|z^{\tau}_s|^{1.5-\epsilon-2\delta})}dyd\sigma{\big |}<c'<\infty
\end{array}
\end{equation}
for a finite constant $c'$, and where we may choose $\delta\in (0,1)$. Here, we may use the elliptic integral estimate
\begin{equation}
\int_{B^n_x}\frac{dy}{|x-y|^a|y|^b}\leq \max\left\lbrace  c|y|^{n-a-b},c\right\rbrace ,
\end{equation} 
for some finite constants $\tilde{c},c>0$ and where $B^n_x$ is a ball of finite radius around $x$. Such an upper nound of an elliptic integral may be obatined by spliiting the integral or by partial integration.
Similar (simpler) estimates hold for the Burgers term and for the initial data convolution term, i.e., we have 
\begin{equation}
\begin{array}{ll}
\sup_{\tau,\uparrow \infty,z^{\tau}_s}{\big |}\int_{t_{in}}^{\tau}\int_{|y|\leq 1}\left( \sum_{j=1}^n \mu_1( w_j+z_j)w_{i,j}\right)(s,y)G_{\nu'}(\tau-s,z^{\tau}_s-y)dyds{\big |}< c'',
\end{array}
\end{equation}
and 
\begin{equation}
\sup_{\tau,z^{\tau}_s}{\Big (}{\big |}\int_{y\in \Omega}w_i(t_{in},y)G_{\nu'}(\tau,z^{\tau}_s-y)dy{\big |}{\Big )}<c'''
\end{equation}
for some finite constants $c'',c'''>0$ by similar considerations as in the case of the Leray projection term. Finally, we consider
\begin{equation}\label{boundw}
\sup_{\tau\uparrow \infty,z^{\tau}_s}{\big |}\int_{t_{in}}^{\tau}\int_{\partial^S Z^{(t_s,x_s)}_{t_1}}w^{\partial^Z}_i(s,y)G_{\nu'}(\tau,z^{\tau}_s;s,y)dyds{\big |}{\Big )}.
\end{equation}
Here recall that $\partial^S Z^{(t_s,x_s)}_{t_1}$ denotes the spatial boundary of the cylinder $Z^{(t_s,x_s)}_{t_1}$ with spatial basis $\Omega$. Recall that we may assume that $0=x_s=z^{\tau}_s$ for all $\tau\in (0,\infty)$. For all points $(\tau,z)$ on the boundary we have
\begin{equation}\label{boundw2}
\begin{array}{ll}
w^{\partial^Z}_i(\tau,z)=-2w_i{\big |}_{\partial^SZ^{(t_s,x_s)}_{t_1}}(\tau,z)+2\int_{\Omega}w_i(t_{in},y)G_{\nu'}(\tau,z;t_{in},y)dy\\
\\
-2(N^i\ast G_{\nu})(\tau,z)+\sum_{k=1}^{\infty}\int_0^{\tau}\int_{\Omega}{\Big (}-2w_i{\big |}_{\partial^S Z^{(t_s,x_s)}_{t_1}}(\sigma,\xi)+\\
\\
2\int_{\Omega}w_i(t_{in},y)G_{\nu'}(\sigma,\xi;t_{in},y)dy-2(N^i\ast G_{\nu'})(\sigma,\xi){\Big )} G^{k}_{\nu'}(\tau,z,\sigma,\xi)d\xi d\sigma,
\end{array}
\end{equation}
where we note that
\begin{equation}\label{vwbd}
G^{k}_{\nu'}=G^{\mu,k}_{\nu},~\mbox{and}~v_i{\big |}_{S_K}(t,x)=w_i(\tau,z){\big |}_{\partial^S Z^{(t_s,x_s)}_{t_1}}.
\end{equation}
Here recall that $v_i{\big |}_{S_K}$ denotes the restriction of the velocity component to the spatial boundary $S_K$ of the cone $K^{(t_s,x_s)}_{j,t_1}$.
According to Hopf's result we have $v_i\in L^{\infty}\left([t_1,t_s],L^2\left({\mathbb T}^n\right)\right)   $ , such that for all $(t,x)\in K^{(t_s,x_s)}_{j,t_1}$ we have
\begin{equation}
{\Big |}v_i{\big |}_{S_K}(t,x){\Big |}\leq \frac{c}{|x-x_s|^{\lambda}}~\mbox{a.s.},
\end{equation}
for some  $0\leq \lambda < 1.5$ and some finite constant $c>0$. Shifting spatial coordinates if necessary, we may assume $x_s=0$. We may write the boundary of the cone $K^{(t_s,x_s)}_{j,t_1}$
\begin{equation}
S_K:=\left\lbrace(t,x)|t\in (t_1,t_s)~\&~|x-x_s|= t_s-t\right\rbrace ,
\end{equation}
with $t_1\in U_j$. This boundary may be written in polar coordinates in order to obtain a simple description of the spatial boundary of the corresponding cylinder. For $r_0:=t_s-t_1$ and $x_s=0$ and $\rho=1$ we have $|x-x_s|=|x|=(t_s-t)^{\rho}|z|=(t_s-t)|z|$ such that the boundary of the corresponding cylinder  is described by
\begin{equation}
\partial^S Z^{(t_s,x_s)}_{t_1}:=\left\lbrace (\tau,z)| |z|=r_0~\&~\tau\geq t_{in}\right\rbrace. 
\end{equation}
We have $x=(t_s-t)^{\rho}z$ with $\rho=1$, and 
\begin{equation}\label{gg}
\begin{array}{ll}
{\big |}G_{\nu'}(\sigma,y){\big |}_{\partial^S Z^{(t_s,x_s)}_{t_1}}
={\Big |}\frac{1}{\sqrt{4\pi \nu' \sigma}^n}
\exp\left(-\frac{|y|^2}{4\nu' \sigma}\right) {\Big |}_{\partial^S Z^{(t_s,x_s)}_{t_1}}\\
\\
={\Big |}\frac{1}{\sqrt{4\pi \nu' \sigma}^n}
\exp\left(-\frac{|r_0|^2}{4\nu' \sigma}\right) {\Big |}.
\end{array}
\end{equation}
Hence, using (\ref{vwbd}), and (\ref{gg}), with $n=3$ and $x_s=0$ we have 
\begin{equation}
\begin{array}{ll}
\sup_{\tau>t_{in}}{\Big |}w_i(\tau-\sigma,0-y){\big |}_{\partial^S Z^{(t_s,x_s)}_{t_1}}\ast G_{\nu}(\sigma,y){\Big |}\\
\\
\leq
 \int_{\partial^SZ^{(t_s,x_s)}_{t_1}}\frac{c}{|r_0|^{\lambda}(t_s-t(\sigma))^{\lambda}}
 {\Big |}\frac{1}{\sqrt{4\pi \nu \sigma}^n}
\exp\left(-\frac{|r_0|^2}{4\nu' \sigma}\right) {\Big |}dy\leq c'<\infty
\end{array}
\end{equation}
for some finite $c,c^*,c'$ and where we may choose $\delta \in (0.5,1)$. The higher order boundary terms can be estimated similarly.
In the case $v_i\in H^{1-\epsilon},~1\leq i\leq n$ for small $\epsilon >0$ we have observed that Lemma \ref{pressure} still holds. Furthermore, all estimates go through straightforwardly. Finally, note that  $v_i(t_1,.)\in C^{\infty}\left({\mathbb T}^n\right)$ such that the upper bounds can be constructed independently of $x_s$, i.e., the spatial component of a possible singularity $(t_s,x_s)$ and  the related cone $K^{(t_s,x_s)}_{t_1}$.  

\section{Proof of theorem \ref{thmsing2}}
First we describe the proof plan. The result that there exists a left continuous extension of a Hopf-Leray solution can be sharpened by consideration of spatial derivatives using spatial symmetry of first order spatial derivatives of the Gaussian together with local Lipschitz continuity of the Leray projection term. This gives the estimates for the crucial Leray projection term convoluted with the first order derivative of the Gaussian. The convoluted initial value term and the convoluted Burgers term have similar estimates a fortiori. In a second step we consider upper bound estimates for spatial derivatives of the boundary  terms.
This argument (given below in detail) implies that for any possible time $t_s$ of a singularity $(t_s,x_s)\in S$ (recall that $S$ is the set of possible singularities predicted by the CKN-theory) we have (for some finite constant $C$)
\begin{equation}
\sup_{t\in [t_1,t_s]}{\big |}v_i(t,.){\big |}_{H^2\cap C^2}\leq C<\infty,
\end{equation}
where $[t_1,t_s)\subset U_j$,and $U_j$ is an open interval where a given Leray Hopf solution has full regularity. Furthermore, within a fixed time interval $[0,T]$ for some arbitrary time horizon $T>0$ the upper bound constants used do not depend on the specific location of singularities at $(t_s,x_s)$, i.e., for any sequence $(t^k_s)$ such that $(t^k_s,x^k_s)\in S$, $T^k_s\downarrow t_s$ as $k\uparrow \infty$, the  finite upper bound constants $C_{t^k_s}$ with ${\big |}v_i(t^k_s,.){\big |}_{H^2\cap C^2}\leq C_{t^k_s}$ have a common upper bound such that $\sup_{k\uparrow \infty}C_{t^k_s}\leq C<\infty.$ 
 
First we prove the existence of a regular spatial left-continuous extension. The latter task is obtained by proving a regular left-continuous extension for the comparison function $w_i,~1\leq i\leq n$, together with a proof of regularity transfer from the function $w_i,~1\leq i\leq n$ to $v_i,~1\leq i\leq n$.

We first consider this regularity transfer from $w_i,~1\leq i\leq  n$ to $v_i,~1\leq i\leq n$ assuming full regularity of the former function and Lipschitz continuity of the gradient of transformed pressure function $p^w_i,~1\leq i\leq n$, i.e., Lipschitz continuity of the transformed Leray projection term.  
This regularity transfer is not trivial (cf. the property of the transformation in (\ref{deralpha}) above). Note that for all $1\leq j\leq n$ we have
\begin{equation}\label{deralpha*}
{\big |}v_{i,j}(t,.){\big |}\leq {\big |}w_{i,j}(\tau,.){\big |}\frac{1}{(t_s-t)^{\rho}},
\end{equation} 
where we again choose $\rho=1$ in the following for simplicity, i.e., in order to have solution representations in term of convolutions with standard Gaussians at hand. For the regularity transfer the Leray projection term is crucial, and we consider this term next. Since $\tau=\frac{t}{t_s-t}$, or $t=t_s\frac{\tau}{1+\tau}$, we have
\begin{equation}\label{piw}
p_{,i}=p^w_{,i}\frac{dz_i}{dx_i}=p^w_{,i}\frac{1}{(t_s-t)}=p^w_{,i}\frac{\tau}{t}=p^w_i(1+\tau)\frac{1}{t_s}.
\end{equation}
We shall observe below that ${\big |}p^w_{,i}(\tau,.){\big |}\leq \frac{C}{1+\tau}$ (even a stronger decay holds).
Anyway, even by the first two relations in (\ref{piw}) we have that regularity of $w_i,~1\leq i\leq n$ and (\ref{condw}) implies that 
\begin{equation}
{\Bigg |}L^i_{{\mathbb T}^n} 
\left( \sum_{j,m=1}^n\left( v_{m,j} v_{j,m} \right) \right)(s,x){\Bigg |}\leq \frac{C}{(t_s-t)^{\rho}},
\end{equation}
which is not integrable for $\rho\geq 1$. Nevertheless, we consider $\rho=1$ such that we have the Gaussian $G^{\mu}_{\nu}=G^1_{\nu}$ in (\ref{wrep}). Note that for $\rho=1$ we have $G^{1}_{\nu}=G_{\nu'}$ with $\nu'=\frac{\nu}{t_s}$, where we use the assumption that $t_s>0$. The function $v_i,~1\leq i\leq n$ is defined on the whole torus, say on $[-0.5,0.5]^n$ with periodic boundary conditions, but we can treat it formally as an initial-boundary value problem with artificial boundaries
\begin{equation}
\partial^S\left( [t_1,t_s]\times {\mathbb T}^n\right):=
\left\lbrace (t,x)|t\in [t_1,t_s]~\&~x_i=-0.5\mbox{ or }x_i=0.5,~ 1\leq i\leq n\right\rbrace.
\end{equation}
Let us explain why the Leray projection term is crucial for the regularity transfer from the function $w_i,~1\leq i\leq n$. 
For the first order spatial derivatives we have the representation 
\begin{equation}\label{wrepderv}
\begin{array}{ll}
 v_{i,j}(t_s,x)=\int_{\left\lbrace y|(t_1,y)\in Z^{(t_s,x_s)}_{t_1}\right\rbrace }v_i(t_1,y)G_{\nu,j}(t_s-t_1,x-y)dy\\
\\
-\int_{[t_1,t_s]\times {\mathbb T}^n}\left( \sum_{j=1}^n ( v_jv_{i,j}\right)(s,y)G_{\nu,j}(t_s-s,x-y)dyds\\
\\ +\int_{[t_1,t_s]\times {\mathbb T}^n}L_{{\mathbb T}^n} 
\left( \sum_{l,m=1}^n\left( v_{m,l} v_{l,m} \right) \right)(s,y)G_{\nu,j}(t_s-s,x-y)dyds\\
\\
+\int_{t_1}^{t_s}\int_{\partial^S \left( [t_1,t_s]\times {\mathbb T}^n\right) }v^{\partial^T}_i(s,y)G_{\nu,j}(\tau,x;s,y)dyds.
\end{array}
\end{equation}
Here, the boundary term $v^{\partial^T}_i(s,y)$ is given by an analogous formula as $w^{\partial^Z}_i$ before.
Furthermore, on the right side of (\ref{wrepderv}) concerning the leading terms only the Burgers term and the Leray projection term involve first order derivatives of the velocity components $v_{i},¸1\leq i\leq n$ such that we have to deal with the regularity loss expressed in (\ref{deralpha*}) when we pass form $w_{i,j}$ to $v_{i,j}$. This is different for the other terms, where we may use $v_i(t,x)=w_i(\tau,z)$, and this is true for the higher order terms of the expansion of  $v^{\partial^T}_i$ as well, of course. The regularity for the latter higher boundary terms follows form the regularity of the boundary terms for the transformed functions  $w^{\partial^Z}_i$ straightforwardly. Note here that the nonlinear terms (the terms abbreviated by $N^i$ in the representation for $w^{\partial^Z}_i$ in (\ref{boundw}) and (\ref{boundw2})) are convoluted twice, and can estimated straightforwardly for $\mu=1$ if a suitable estimate for the Leray projection term is at hand.
Hence, the Leray projection term is indeed crucial. First note that the assumption of a full regularity of the comparison functions $w_i$ on a time interval $[t_1,t_s]$ implies that we have spatial Lipschitz continuity of the Leray projection term, i.e., we  obtain this spatial Lipschitz continuity from spatial Lipschitz continuity of $p^w_{,i},~1\leq i\leq n$, cf. below. This means that for all $x$ we have the relation
\begin{equation}
{\big |}L^i_{{\mathbb T}^n} 
\left( \sum_{j,m=1}^n\left( v_{m,j} v_{j,m} \right) \right)(s,x-y)-L^i_{{\mathbb T}^n} 
\left( \sum_{j,m=1}^n\left( v_{m,j} v_{j,m} \right) \right)(s,x-y'){\big |}\leq l|y-y|',
\end{equation} 
for a Lipschitz constant which is independent of $x$, and  for all $s\in [t_1,t_s]$ and $1\leq i\leq n$ (inherited from spatial Lipschitz continuity of the Leray projection term $L^i_{{\mathbb T}^n} \left( \sum_{j,m=1}^n\left( w_{m,j} w_{j,m} \right) \right)(s,x),~1\leq i\leq n$ proved in detail below).
We use this Lipschitz continuity together with the symmetry of the first order derivative of the Gaussian. In this context for $y=(y_1,\cdots,y_n)$ let
\begin{equation}
y^{-j}=(y^{-j}_1,\cdots,y^{-j}_n),~y^{-j}_j=-y_j,~y^{-j}_k=y_k~\mbox{for $k\neq j$.}
\end{equation}
We get (for $n=3$)
\begin{equation}\label{lipestest}
\begin{array}{ll}
\int_{t_1}^{t_s}\int_{[-0.5,0.5]^3}{\Big |}L^i_{{\mathbb T}^n} 
\left( \sum_{j,m=1}^n\left( v_{m,j} v_{j,m} \right) \right)(t_s-s,.-y) G_{\nu,j}(s,y){\Big |}dyds\\
\\
\leq \int_{t_1}^{t_s}\int_{[-0.5,0.5]^3,y_j\geq 0}{\Bigg |}{\Big (}L^i_{{\mathbb T}^n} 
\left( \sum_{j,m=1}^n\left( v_{m,j} v_{j,m} \right) \right)(t_s-s,.-y)\
\\
-L^i_{{\mathbb T}^n} 
\left( \sum_{j,m=1}^n\left( v_{m,j} v_{j,m} \right) \right)(t_s-s,x-y^j){\Big )}\frac{y_j}{2(t_s-s)\sqrt{4\pi \nu s}^n}\exp\left(-\frac{|y|^2}{4\nu (t_s-s)} \right){\Bigg |}dyds\\
\\
\leq \int_{t_1}^{t_s}\int_{[-0.5,0.5]^3,y_j\geq 0}{\Bigg |}\frac{C}{(t_s-t)^{\mu}}\frac{l|y|^2}{(t_s-s)\sqrt{4\pi \nu s}^n}\exp\left(-\frac{|y|^2}{4\nu (t_s-s)} \right){\Bigg |}dyds
\\
\\
\leq \frac{C^*}{(t_s-t)^{\delta}}, \mbox{ for $\delta \in (0,1)$,}
\end{array}
\end{equation}
and for some finite constants $C,C^*$, and where we use the standard pointwise estimate
\begin{equation}
\begin{array}{ll}
{\big |}G^1_{\nu,j}(t-s,y){\big |}\sim{\big |}\frac{|y|}{t_s-s}G^1_{\nu}(t-s,y){\big |} 
={\Big |}\frac{|y|}{(t_s-s)\sqrt{4\pi \nu (t-s)}^n}\exp\left(-\frac{|y|^2}{4\nu (t_s-s)} \right){\Big |}\\
\\
\leq \frac{c}{(4\pi\nu (t_s-s))^{\delta}}|y|\left( |y|^2\right)^{\delta -{n}/2-1} \left( \frac{|y|^2}{4\pi\nu (t_s-s)}\right)^{n/2+1-\delta} \exp\left(-\frac{|y|^2}{4\nu (t_s-s)} \right)\\
\\
\leq \frac{C}{(4\pi\nu (t_s-s))^{\delta}|y|^{n+1-2\delta}}.
\end{array}
\end{equation}
Note that the latter upper bound is only integrable for $\delta \in (0.5,1)$, and it is the spatial Lipschitz continuity of the Leray projection term which allows to get the upper bound in (\ref{lipestest}) with $\delta \in (0,1)$. It follows that for all $t\in [t_1,t_s]$
\begin{equation}
\int_{t_1}^{(.)}\int_{[-0.5,0.5]^3}L^i_{{\mathbb T}^n} 
\left( \sum_{j,m=1}^n\left( v_{m,j} v_{j,m} \right)(.-s,.-y) \right) G_{\nu,j}(s,y)dyds\in L^p
\end{equation}
for all $p>0$ (note that have passed the barrier $p=\frac{1}{4}$). Iteration of this argument (and application to spatial derivatives) leads to full regularity transfer.
Hence it remains to prove that we have a)  full regularity of the comparison function $w_i,¸1\leq i\leq n$, and b) that spatial Lipschitz continuity transfers from $w_i, 1\leq i\leq n$ to $v_i,¸1\leq i\leq n$.
For the first order derivatives $w_{i,j}$ we have a the classical representation
\begin{equation}\label{wrepder}
\begin{array}{ll}
 w_{i,j}(\tau,z)=\int_{\left\lbrace y|(t_1,y)\in Z^{(t_s,x_s)}_{t_1}\right\rbrace }w_i(t_{in},y)G^{\mu}_{\nu,j}(\tau,z-y)dy\\
\\
-\int_{Z^{(t_s,x_s)}_{t_1,\tau}}\left( \sum_{j=1}^n (\mu_1 w_j+z_j)w_{i,j}\right)(s,y)G^{\mu}_{\nu,j}(\tau-s,z-y)dyds\\
\\ +\int_{Z^{(t_s,x_s)}_{t_1,\tau}}\mu_1(s)L^i_{{\mathbb T}^n} 
\left( \sum_{l,m=1}^n\left( w_{m,l} w_{l,m} \right) \right)(s,y)G^{\mu}_{\nu,j}(\tau-s,z-y)dyds\\
\\
+\int_{t_{in}}^{\tau}\int_{\partial^S Z^{(t_s,x_s)}_{t_1,\tau}}w^{\partial^Z}_i(s,y)G^\mu_{\nu,j}(\tau,z;s,y)dyds,
\end{array}
\end{equation}
where $Z^{(t_s,x_s)}_{t_1,\tau}$ denotes the cylinder cut off at $\tau>t_i$.
Hence, for all $1\leq j\leq n$
\begin{equation}\label{wrep2der}
\begin{array}{ll}
\sup_{\tau\uparrow \infty,z^{\tau}_s} {\big |}w_{i,j}(\tau,z^{\tau}_s){\big |}\\
\\
 \leq \sup_{(\tau,x)\in Z^{(t_s,x_s)}_{t_1}}{\Big (}{\big |}\int_{\left\lbrace y|(t_1,y)\in Z^{(t_s,x_s)}_{t_1}\right\rbrace }w_i(t_{in},y)G^{\mu}_{\nu,j}(\tau-t_{in},z-y)dy{\big |}\\
\\
+{\big |}\int_{Z^{(t_s,x_s)}_{t_1}}\left( \sum_{j=1}^n (\mu_1 w_j+z_j)w_{i,j}\right)(s,y)G^{\mu}_{\nu,j}(\tau-s,z-y)dyds{\big |}\\
\\ 
+{\big |}\int_{Z^{(t_s,x_s)}_{t_1}}\mu_1(s)L^i_{{\mathbb T}^n} 
\left( \sum_{l,m=1}^n\left( w_{m,l} w_{l,m} \right) \right)(s,y)G^{\mu}_{\nu,j}(\tau-s,z-y)dyds{\big |}\\
\\
+{\big |}\int_{t_{in}}^{\tau}\int_{\partial^S Z^{(t_s,x_s)}_{t_1}}w^{\partial^Z}_i(s,y)G^\mu_{\nu ,j}(\tau,z^{\tau}_s;s,y)dyds{\big |}{\Big )}.
\end{array}
\end{equation}
We estimate the Leray projection term and remark that the the boundary term can be estimated using the ideas of the previous section. Similar estimates hold also for the Burgers term and the initial value term a fortiori. Consider the cone $K^{(t_s,x_s)}_{t_1}$ associated to a possible singularity $(t_s,x_s)\in S$ and consider a Leray-Hopf solution $v_i,~1\leq i\leq n$ on this cone. If $v_i\in H^1$ then 
\begin{equation}
{\Bigg |}L^i_{{\mathbb T}^n} 
\left( \sum_{j,m=1}^n\left( v_{m,j} v_{j,m} \right) \right)(t,x){\Bigg |}\leq \frac{C}{(t_s-t)^{\delta}|x_s-x|^{\frac{3}{2}-\epsilon}},
\end{equation}
for some finite constant $C>0$ and some $\delta \in [0,0.5)$ and small $\epsilon>0$. If $v_i\in H^{1-\epsilon'}$ for some small $\epsilon'$, then 
\begin{equation}
{\Bigg |}L^i_{{\mathbb T}^n} 
\left( \sum_{j,m=1}^n\left( v_{m,j} v_{j,m} \right) \right)(t,x){\Bigg |}\leq \frac{C}{(t_s-t)^{\delta}|x_s-x|^{\frac{3}{2}+\epsilon}},
\end{equation}
for some finite constant $C>0$ and some $\delta \in [0,0.5+\epsilon]$ and small $\epsilon>0$. We estimate the Leray projection term of the transformed equation in the latter case. Shifting spatial coordinates we may assume that $x_s=0$ and have 
\begin{equation}
{\Bigg |}L^i_{{\mathbb T}^n} 
\left( \sum_{j,m=1}^n\left( v_{m,j} v_{j,m} \right) \right)(t,x){\Bigg |}\leq \frac{C}{(t_s-t)^{\delta}|x|^{\frac{3}{2}+\epsilon}}.
\end{equation}
The left side denotes the first spatial derivative of the pressure with respect to the argument $x_i$. Recall that (with $\mu=1$)
\begin{equation}
p_{,i}=p^w_{,i}\frac{1}{(t_s-t)},~\mbox{or}~p^w_{,i}=(t_s-t)p_{,i}.
\end{equation}
Recall that for the first order spatial derivatives of the Gaussian we have the standard estimate
\begin{equation}
{\big |}G_{\nu,j}(\sigma,y){\big |}\leq \frac{c}{|\sigma|^{\delta}|y|^{n+1-2\delta}}.
\end{equation}
Note again that, alternatively, since $\sigma\geq t_{in}$, we have a simple estimate
\begin{equation}
{\big |}G^1_{\nu,j}(\sigma,y){\big |}\lesssim \frac{1}{\sqrt{\sigma}^3|y|},
\end{equation}
where we may use $\frac{|y|^2}{\sigma}G^1_{\nu}(\sigma,y)=\frac{|y|^2}{\sigma}G^1_{2\nu}(\sigma,y)G^1_{2\nu}(\sigma,y)\leq C{\big |}G^1_{2\nu}(\sigma,y){\big |}$ for some finite constant $C>0$. Hence even stronger estimates than the following hold.
Recall that Leray projection term is the convolution of the latter two terms times $\mu_1=(t_s-t)/t_s$. 
A similar reasoning as in the last section leads (in case $n=3$) to the upper bound
\begin{equation}
\begin{array}{ll}
\sup_{\tau\uparrow \infty}{\big |}\int_{t_1}^{\tau}\int_{|y|\leq 1}\mu_1(t(\sigma))L^i_{{\mathbb T}^n} 
\left( \sum_{j,m=1}^n\left( w_{m,j} w_{j,m} \right) \right)(\sigma,y)\times\\
\\
\times G^1_{\nu,j}(\tau-\sigma,z-y)dyd\sigma{\big |}\\
\\
\leq \sup_{\tau\uparrow \infty}{\big |}\int_{t_1}^{\tau}\int_{|y|\leq 1}
\frac{c}{(1+\sigma)^{2+\delta_1}
(|y|^{1.5+\epsilon})}\frac{c}{|\tau-\sigma|^{\delta}|z-y|^{4-2\delta}}dyd\sigma{\big |}\\
\\
\leq \sup_{\tau\uparrow \infty}{\big |}\int_{t_1}^{\tau}
\frac{c}{(1+\tau)^{1+\delta_1+\delta}
(|z|^{2.5+\epsilon-2\delta})}d\sigma{\big |}<\frac{c'}{|z|^{0.5+\epsilon'}}.
\end{array}
\end{equation}
for some appropriate constant $c_0$ and small $\epsilon'>0$. This means that we have for dimension $n\geq 2$
\begin{equation}
\begin{array}{ll}
\sup_{\tau\uparrow \infty}{\big |}\int_{t_1}^{\tau}\int_{|y|\leq 1}\mu_1(t(\sigma))L^i_{{\mathbb T}^n} 
\left( \sum_{j,m=1}^n\left( w_{m,j} w_{j,m} \right) \right)(\sigma,y)\times\\
\\
\times G^1_{\nu,j}(\tau-\sigma,.-y)dyd\sigma{\big |}\in L^{3-\epsilon^*}
\end{array}
\end{equation}
for small $\epsilon^*>0$. As the Burgers term and the initial value term have stronger regularity, and assuming that the boundary term has the same regularity at least (which is indeed true and checked below) we conclude that
\begin{equation}
w^{(0)}_{i,j}(\tau,.):=w_{i,j}(\tau,.)\in L^3~\mbox{for $\tau\geq t_{in}$.}
\end{equation}
We may use this information as an input in the estimate in (\ref{wrep2der}) above, i.e., we may set up an iterative scheme for $k\geq 1$ of the form
\begin{equation}\label{wrep2der}
\begin{array}{ll}
\sup_{\tau\geq t_{in}} {\big |}w^{(k)}_{i,j}(\tau,.){\big |}\\
\\
 \leq \sup_{\tau\geq t_{in}}{\Big (}{\big |}\int_{\left\lbrace y|(t_1,y)\in Z^{(t_s,x_s)}_{t_1}\right\rbrace }w^{(k-1)}_i(t_{in},y)G^{\mu}_{\nu,j}(\tau,z-y)dy{\big |}\\
\\
+{\big |}\int_{Z^{(t_s,x_s)}_{t_1}}\left( \sum_{j=1}^n (\mu_1 w^{(k-1)}_j+x_j)w^{k-1)}_{i,j}\right)(s,y)G^{\mu}_{\nu,j}(\tau-s,,-y)dyds{\big |}\\
\\ 
+{\big |}\int_{Z^{(t_s,x_s)}_{t_1}}\mu_1(s)L^i_{{\mathbb T}^n} 
\left( \sum_{l,m=1}^n\left( w^{(k-1)}_{m,l} w^{(k-1)}_{l,m} \right) \right)(s,y)G^{\mu}_{\nu,j}(\tau-s,.-y)dyds{\big |}\\
\\
+{\big |}\int_{t_{in}}^{\tau}\int_{\partial^S Z^{(t_s,x_s)}_{t_1}}w^{\partial^Z,(k-1)}_i(s,y)G^\mu_{\nu ,j}(\tau,z;s,y)dyds{\big |}{\Big )},
\end{array}
\end{equation}
where for the boundary term we have
\begin{equation}
w^{\partial^Z,(0)}_i=w^{\partial^Z}_i,
\end{equation}
and such that $w^{\partial^Z,(k)}_i$ is defined recursively and analogously. The functions which define the boundary terms inherit regularity $w^{(k-1)}_i$ known from the previous step and it is easy to check that the upper bounds which hold for the Leray projection term are a fortiori upper bounds for the boundary terms (you may even use the fact that we are on the boundary of a cylinder where the basis is a ball of positive radius around $x_s=0$).
We shall observe that gain spatial regularity of one order at least at each iteration step. Some embedding results may be used here (cf. \cite{N}). We have
\begin{equation}
g\in H^{s,p}~\mbox{ iff }~\Lambda^sg\in L^p,
\end{equation}
where
\begin{equation}
{\cal F}\left( \Lambda^sg\right) =(1+\xi^2)^{s/2}{\cal F}\left(g \right)(\xi),
\end{equation}
and
\begin{equation}
H^{r,q}\subset H^{s,p}~\mbox{ iff }~\frac{1}{q}-\frac{1}{p}=\frac{r-s}{n}.
\end{equation}
As usual and as before we drop the second superscript in case of $L^2$-theory. 
For $n=3$ we start with $w_{i,j}(\tau,.)\in L^3$ and get $w_{i,j}(\tau,.)\in H^{0.5}$.
The functions $w_i,~1 \leq i\leq n$ are defined on a cylinder (not on a torus), but we can adopt dual Sobolev spaces obviously (the formal definition may be supplemented by the reader). Recall that
\begin{equation}
w_i(\tau,.)\in H^s~\mbox{ iff }~\sum_{\alpha \in{\mathbb Z}^n}|w_{i\alpha}(\tau)|^2\left\langle \alpha\right\rangle^{2s}< \infty, 
\end{equation}
If $w^{(0)}_{i,j\alpha},~\alpha\in {\mathbb Z}^n$ denote the $\alpha$-modes of $w^{(0)}_{i,j}(\tau,.)$ then 
\begin{equation}
{\big |}w^{(0)}_{i,j\alpha}(\tau,.){\big |}\leq \frac{c}{\left\langle \alpha\right\rangle^{2+\epsilon}},~\mbox{ and }~{\big |}w^{(0)}_{i,\alpha}(\tau,.){\big |}\leq \frac{c}{\left\langle \alpha\right\rangle^{3+\epsilon}}
\end{equation}
for some small $\epsilon >0$, where we recall
\begin{equation}
 \left\langle \alpha\right\rangle :=\left(1+|\alpha|^2 \right)^{1/2}. 
\end{equation}
Furthermore, the Poisson elimination equation for the pressure is the same for $w_i$ as in original coordinates, such that we have 
\begin{equation}\label{navode200first222}
\begin{array}{ll}
p^{w,(0)}_{,i}(\tau,z)=p^w_{,i}(\tau,z)=\sum_{\alpha\in {\mathbb Z}^n}p^w_{\alpha ,i}\exp\left( 2\pi i\alpha z\right)
\\
\\
=\sum_{\alpha\in {\mathbb Z}^n}2\pi i\alpha_i1_{\left\lbrace \alpha\neq 0\right\rbrace}\frac{\sum_{j,k=1}^n\sum_{\gamma\in {\mathbb Z}^n}4\pi^2 \gamma_j(\alpha_k-\gamma_k)w_{j\gamma}w_{k(\alpha-\gamma)}}{\sum_{i=1}^n4\pi^2\alpha_i^2}\exp\left( 2\pi i\alpha z\right).
\end{array} 
\end{equation}
It follows that
\begin{equation}
p^{w,(0)}_{,i}(\tau,.)=p^w_{,i}(\tau,.)\in H^{1-\epsilon}
\end{equation}
for small $\epsilon >0$. Iterating this argument it follows that for all $k\geq 1$
\begin{equation}
w^{(k)}_{i,j}(\tau,.):=w_{i,j}(\tau,.)\in H^{k-\epsilon}~\mbox{for $\tau\geq t_{in}$.}
\end{equation}
for small $\epsilon >0$. Hence we have proved a). For b), i.e., the decay of $p^w_{,i}$ for large $\tau$ of at least order $1$ we start with
\begin{equation}
p^w_{,i}=L^i_{{\mathbb T}^n}\left(\sum_{l,m}w_{l,m}w_{m,l} \right),
\end{equation}
and go back to the representation of the first derivatives of $w_i$ in (\ref{wrepder}).
We use again an 'iterative scheme' based on an equivalent representation. Here iterative scheme means that we have indeed a fixed point scheme but iterate regularity considerations with respect to this fixed point according to the scheme. We have
\begin{equation}\label{wrepder2}
\begin{array}{ll}
w^{(k)}_{i,j}(\tau,z)= w_{i,j}(\tau,z)=\int_{\left\lbrace y|(t_{in},y)\in Z^{(t_s,x_s)}_{t_1,\tau}\right\rbrace }w^{(k-1)}_i(t_{in},z-y)G^{\mu}_{\nu,j}(\tau-t_{in},y)dy\\
\\
-\int_{Z^{(t_s,x_s)}_{t_1,\tau}}\left( \sum_{j=1}^n (\mu_1 w^{(k-1)}_j+x_j)w^{(k-1)}_{i,j}\right)(\tau-s,z-y)G^{\mu}_{\nu,j}(s,y)dyds\\
\\ +\int_{Z^{(t_s,x_s)}_{t_1,\tau}}\mu_1(s)L^i_{{\mathbb T}^n} 
\left( \sum_{l,m=1}^n\left( w^{(k-1)}_{m,l} w^{(k-1)}_{l,m} \right) \right)(\tau-s,z-y)G^{\mu}_{\nu,j}(s,y)dyds\\
\\
+\int_{t_{in}}^{\tau}\int_{\partial^S Z^{(t_s,x_s)}_{t_1}}w^{\partial^Z,(k-1)}_i(\tau-s,z-y)G^\mu_{\nu,j}(s,y)dyds,
\end{array}
\end{equation}
where $w^{(k-1)}_{m,l}=w_{m,l}$. Consider again $\mu=1$. 
We observe a gain of regularity at each iteration step. As observed above for $\mu=1$ and $n=3$ we may us for $s\geq t_{in}$ the estimate ${\big |}G^{1}_{\nu,j}(s,y){\big |}\leq \frac{c}{\sqrt{s}^3|y|}$ for some $c>0$. All moduli of functions  $w^{(0)}_{m,l}=w_{m,l}$, $w^{(0)}_{m}=w_{m}$, and $w^{\partial^Z,(0)}_l=w^{\partial^Z}_l$ have a finite constant upper bound $C>0$ for all $1\leq l,m\leq n$, such that we get an upper bound for ${\big |}w^{(1)}_{m,l}{\big |}\leq \frac{C}{\sqrt{s}}$, ${\big |}w^{(1)}_{m}{\big |}\leq \frac{C}{\sqrt{s}}$, and ${\big |}w^{\partial^Z,(1)}_l{\big |}\leq \frac{C}{\sqrt{s}}$ for $s\geq t_{in}$ after one iteration step (note that the terms with coefficients $\mu_0\sim \frac{1}{1+\tau}\sim\mu_1$ have even a stronger decay by one order at this first iteration step. After the second iteration we have the desired decay with respect to $\tau$. This proves b). Finally we remark that the constants inherited from CKN-theory are all global in a given domain with arbitrary finite time interval $[0,T]$ with arbitrary finite time horizon $T>0$ and do not depend on $(t_s,x_s)\in S$. Hence, we have indeed a regular extension of the Leray-Hopf solution.  

\section{Conclusion}

The Cafarelli-Kohn-Nirenberg theory is one of the powerful tools of the last century, which allows to derive conclusions about existence, regularity, and singular behaviour for weak function spaces for a considerable class of fluid models. It has shown that the treatment of the Navier Stokes equation is hard from the perspective of weak function spaces. Conclusions can be obtained which go beyond statements of global regularity which assume more regular data. We have stated only a few in this paper, but the method outlined here my be used in order to investigate the asymptotic singular behavior near weak data as $t_s\downarrow 0$ etc..   In the appendix an alternative short argument is given which reduces the global regularity and existence  problem to a local time contraction results for $H^m\cap C^m$ data for $m\geq 2$.  We think that the incompressible Navier Stokes equation problem belongs to a huge class of evolution problems which have a global regular or global smooth solution branch. Many equations of this class may have singular solutions, but in case of the incompressible Navier Stokes equation a global regular solution branch $v\in C^0\left([0,T],H^m\cap C^m\right)$ for arbitrary $T>0$ is well-known to be unique, i.e., if $\tilde{v}_i,~1\leq i\leq n$ is another solution of the incompressible Navier Stokes equation, then we have
\begin{equation}
{\big |}\tilde{v}(t)-v(t){\big |}^2_{L^2}\leq {\big |}\tilde{v}(0)-v(0){\big |}^2_{L^2}
\exp\left(C\int_0^t\left( {\big |}v(s){\big |}^p_{L^4}+{\big |}v(s){\big |}^2_{L^4}\right)ds  \right) 
\end{equation}
where $C>0$ with $p=8$ in dimension $n=3$.
From the perspective of CKN theory after small time $t_0>0$ we can close the gap between
 $$v_i\in L^{\infty}\left([t_0,T],L^2\left({\mathbb T}^n\right)  \right)\cap L^2\left( [0,T],H^{1-\epsilon}\left({\mathbb T}^n\right) \right)\cap  L^{8/3}\left(\left[0,T\right],L^4\left({\mathbb T}^n\right)   \right) \mbox{small $\epsilon$}$$
 and
 $$v_i\in L^{\infty}\left([t_0,T],L^2\left({\mathbb T}^n\right)  \right)\cap L^2\left( [0,T],H^1\left({\mathbb T}^n\right) \right)\cap  L^{8}\left(\left[0,T\right],L^4\left({\mathbb T}^n\right)   \right),$$
and the so-called global regularity and existence problem follows then from local time contraction, if strong data, say $v_i(0,.)\in H^m\cap C^m$, $m\geq 2$ are assumed.

\section{Appendix 1: Comparison to another global regularity argument}

We consider the reduction of the global regular existence problem to a local time contraction result.
This argument is considered for the problem on the whole domain ${\mathbb R}^n$. It is a variation of arguments given elsewhere, and works even without viscosity damping estimates.   The argument can be reformulated for the $n$-torus.
Local time contraction on a short time interval $[t_0,t_0+\Delta]$ for data $v^{\nu}_i(t_0,.),~1\leq i\leq n,~t_0\geq 0$ with $v_i(t_0,.)\in H^m\cap C^m$ for $m\geq 2$ shows that there is a local- time representation of the velocity component functions $v^{\nu}_i,~ 1\leq i\leq n$ of the form

\begin{equation}\label{Navlerayscheme}
\begin{array}{ll}
 v^{\nu}_i=v^{\nu}_i(t_0,.)\ast_{sp}G_{\nu}
-\sum_{j=1}^n \left( v^{\nu}_j\frac{\partial v^{\nu}_i}{\partial x_j}\right) \ast G_{\nu}\\
\\+\left( \sum_{j,m=1}^n\int_{{\mathbb R}^n}\left( \frac{\partial}{\partial x_i}K_n(.-y)\right) \sum_{j,m=1}^n\left( \frac{\partial v^{\nu}_m}{\partial x_j}\frac{\partial v^{\nu}_j}{\partial x_m}\right) (.,y)dy\right) \ast G_{\nu}.
\end{array}
\end{equation}
Using the incompressibility condition
\begin{equation}\label{incomp}
\sum_{j=1}^n\frac{\partial v_j}{\partial x_j}=0,
\end{equation}
we may rewrite the Burgers term, where we have
\begin{equation}
\sum_{j=1}^n\frac{\partial ( v_iv_j)}{\partial x_j}=\sum_{j=1}^nv_j\frac{\partial v_i}{\partial x_j}+v_i\sum_{j=1}^n\frac{\partial v_j}{\partial x_j}= \sum_{j=1}^nv_j\frac{\partial v_i}{\partial x_j}.
\end{equation}
Hence the local representation in (\ref{Navlerayscheme}) may be rewritten in the form
\begin{equation}\label{Navlerayscheme2i}
\begin{array}{ll}
 v^{\nu}_i=v^{\nu}_i(t_0,.)\ast_{sp}G_{\nu}
-\sum_{j=1}^n \left( v^{\nu}_jv^{\nu}_i\right) \ast G_{\nu,j}\\
\\+\left( \int_{{\mathbb R}^n}\left( K_n(.-y)\right) \sum_{j,m=1}^n\left( \frac{\partial v^{\nu}_m}{\partial x_j}\frac{\partial v^{\nu}_j}{\partial x_m}\right) (.,y)dy\right) \ast G_{\nu,i},
\end{array}
\end{equation}
where all nonlinear terms are convolutions with a first order spatial derivative of the Gaussian $G_{\nu}$. We may use the Lipschitz continuity of the Leray projection term function for strong data, which we get close to data at $t_0\geq 0$ by local time contraction in regular space with respect to the norm $\sup_{t\in [t_0,t_0+\Delta]}{\big |}v^{\nu}_i(t,.){\big |}_{H^m\cap C^m}$ for $m\geq 2$. 
For the first order spatial derivatives of the Gaussian we compute
\begin{equation}
\begin{array}{ll}
{\big |}G_{\nu,i}(t,y){\big |} ={\Big |}\frac{-2y_i}{4\pi \nu t}\frac{1}{\sqrt{4\pi\nu t}^n}\exp\left(-\frac{|y|^2}{4\nu t} \right){\Big |}\\
\\
\leq \frac{1}{(4\pi\nu t)^{\delta}|y|}\left( \frac{|y|^2}{4\pi\nu }\right)^{\delta-n/2} \left( |y|^2\right)^{n/2+1-\delta} \exp\left(-\frac{|y|^2}{4\nu t} \right)
\end{array}
\end{equation}
Hence we have for $\delta \in (0,1)$ 
\begin{equation}
{\big |}G_{\nu ,i}(t,y){\big |}\leq \frac{C}{(4\pi \nu t)^{\delta}|y|^{n+1-2\delta}}, 
\end{equation}
where the constant $$C=\sup_{|z|>0}\left( z\right)^{n/2+1-\delta} \exp\left(-z^2\right) >0$$ is independent of $\nu >0$. Similarly we get
\begin{equation}
{\big |}G^r_{\nu ,i}(t,y){\big |}\leq \frac{C}{(4\pi \nu r^2 t)^{\delta}|y|^{n+1-2\delta}}, 
\end{equation}
for the scaled equation with transformation $v^{r,\nu}_i(t,y)=v^{\nu}_i(t,x)$ and $y=rx$. For $1\leq |\beta|\leq m$ we have 
\begin{equation}\label{Navlerayscheme2ir}
\begin{array}{ll}
 D^{\beta}_xv^{r,\nu}_i=D^{\beta}_xv^{r,\nu}_i(t_0,.)\ast_{sp}G^r_{\nu}
-r\sum_{j=1}^n D^{\beta}_{x}\left( v^{r,\nu}_jv^{r,\nu}_i\right) \ast G^r_{\nu,j}\\
\\+r\left( \int_{{\mathbb R}^n}\left( K_{n,i}(.-y)\right) \sum_{j,m=1}^nD^{\gamma}_x\left( \frac{\partial v^{r,\nu}_m}{\partial x_j}\frac{\partial v^{r,\nu}_j}{\partial x_m}\right) (.,y)dy\right) \ast G^r_{\nu,j},
\end{array}
\end{equation}
where for $|\beta|\geq 1$ we have $0\leq |\gamma|\leq m-1$, $\gamma_l+\delta_{lj}=\beta_l$ for all $1\leq l\leq n$, and where we used the convolution rule.
Now consider the comparison function $u^{r,t_0}_i,~1\leq i\leq n$ where for all $y=rx$
\begin{equation}\label{globtimei}
(1+t)u^{r,t_0}_i(s,y)=v^{r}_i(t,y)=v_i(t,x),~s=\frac{t-t_0}{\sqrt{1-(t-t_0)^2}}.
\end{equation}
We may consider this transformation for data $v^r_i(t_0,.)$ on a time interval $t-t_0=0.5$ which corresponds to a time interval of length $\left[0,\frac{1}{\sqrt{3}}\right]$ in terms of the transformed time $s$. 
Next by local time contraction results we have the local representations
\begin{equation}\label{Navlerayusubschemepreia0}
\begin{array}{ll}
u^{r,t_0}_i(s,x)=\int_{{\mathbb R}^n}u^{r,t_0}_i(0,y)G^{\mu,r}_{\nu}(s,x;0,y)dy\\
\\
-\int_{0}^{s}\int_{{\mathbb R}^n}\mu(\sigma)  u^{r,t_0}_{i}(\sigma,y)G^{\mu,r}_{\nu}(s,x;\sigma,y)dyd\sigma\\
\\
-\int_{0}^{s}\int_{{\mathbb R}^n}r\mu^{\tau,2}(\sigma)\sum_{j=1}^n \left( u^{r,t_0}_j u^{r,t_0}_i\right) (\sigma,y)G^{\mu,r}_{\nu,j}(s,x;\sigma,y)dyd\sigma\\
\\
+ \int_{0}^{s}\int_{{\mathbb R}^n} r\mu^{\tau,2}(\sigma)\sum_{j,r=1}^n\int_{{\mathbb R}^n}\left(K_n(z-y)\right)\times\\
\\
\times \sum_{j,l=1}^n\left( \frac{\partial u^{r,t_0}_l}{\partial x_j}\frac{\partial u^{r,t_0}_j}{\partial x_l}\right) (\sigma,y)G^{\mu,r}_{\nu,i}(s,x;\sigma,z)dydzd\sigma,
\end{array}
\end{equation}
and for the multivariate spatial derivatives of order $1\leq |\beta|\leq m$ we have 
\begin{equation}\label{Navlerayusubschemepreia}
\begin{array}{ll}
D^{\beta}_xu^{r,t_0}_i(s,x)=\int_{{\mathbb R}^n}D^{\beta}_xu^{r,t_0}_i(0,y)G^{\mu,r}_{\nu}(s,x;0,y)dy\\
\\
-\int_{0}^{s}\int_{{\mathbb R}^n}\mu(\sigma)  u^{r,t_0}_{i,\beta}(\sigma,y)G^{\mu,r}_{\nu}(s,x;\sigma,y)dyd\sigma\\
\\
-\int_{0}^{s}\int_{{\mathbb R}^n}r\mu^{\tau,2}(\sigma)\sum_{j=1}^n \left( u^{r,t_0}_j u^{r,t_0}_i\right)_{,\beta} (\sigma,y)G^{\mu,r}_{\nu,j}(s,x;\sigma,y)dyd\sigma\\
\\
+ \int_{0}^{s}\int_{{\mathbb R}^n} r\mu^{\tau,2}(\sigma)\sum_{j,r=1}^n\int_{{\mathbb R}^n}\left( K_n(z-y)\right)\times\\
\\
\times \sum_{j,l=1}^n\left( \frac{\partial u^{r,t_0}_l}{\partial x_j}\frac{\partial u^{r,t_0}_j}{\partial x_l}\right)_{,\gamma} (\sigma,y)G^{\mu,r}_{\nu,i}(s,x;\sigma,z)dydzd\sigma,
\end{array}
\end{equation}
where we use the integration of the Burgers term above and we have for $s\in \left[0,\frac{1}{\sqrt{3}} \right] $ and for $k\in \left\lbrace 1,2\right\rbrace $ we have
\begin{equation}\label{mu}
\begin{array}{ll}
\mu =\mu(s)=\frac{\sqrt{1-(t(s)-t_0)^2}^3}{1+t(s)}\geq \frac{3\sqrt{3}}{8(1+T)}=:\mu_0,\\
\\
~r\mu^{\tau, k}:=r(1+t(s))^k\mu=r\sqrt{1-(t(s)-t_0)^2}^3(1+t(s))\\
\\
\leq r(1+T).
\end{array}
\end{equation}
Here, $G^{\mu,r}_{\nu}$ is fundamental solution of the heat equation
\begin{equation}
G^{\mu,r}_{\nu,s}-\mu^1 \Delta G^{\mu,r}_{\nu}=0.
\end{equation}
Now the local solution representation with first order spatial derivatives of the Gaussian has a Levy expansion with leading term upper bound (needed on a compact domain) 
\begin{equation}\label{gmu}
{\big |}G^{\mu,r}_{\nu ,i}(s,y;\sigma,z){\big |}\leq \frac{C}{(\mu_0\nu r^2(s-\sigma))^{\delta}|y-z|^{n+1-2\delta}}, 
\end{equation}
where $C>0$ is a finite constant (dependent only on dimension).
Note that the term $G^{\mu,r}_{\nu ,i}(s,y;\sigma,z)$ is not integrable for $\delta\in \left(0, \frac{1}{2}\right)$ 
but convolutions $l\ast G^{\mu,r}_{\nu ,i}$ with a Lipschitz continuous function $y\rightarrow l_x(y)=l(x-y)$ are, i.e., for such a function with Lipschitz constant $l_0$ there is some finite constant $C>0$ (dependent only on dimension) such that on a ball $B$ of radius $\mu_0\nu r^2$ around the origin the increment over a time interval $[t_0,s]$ of such a term has the upper bound
\begin{equation}\label{uppbound}
\begin{array}{ll}
\int_{t_0}^s\sup_{B}{\big |}l(.-y)\ast G^{\mu,r}_{\nu ,i}(\sigma,y){\big |}\leq 
\int_{t_0}^s\int_B\frac{l_0C}{( \mu_0\nu r^2(\sigma-t_0)^{\delta}|z|^{n-2\delta}}dzd\sigma\\
\\
\leq l_0C(\mu_0\nu r^2)^{\delta} (s-t_0)^{1-\delta},
\end{array}
\end{equation}
where $l_0$ is a Lipschitz constant of $l_x$. In the case of data $v_i(t_0,.)\in H^m\cap C^m$ for $m\geq 2$ at time $t_0\geq 0$ a local time analysis over a msall intervalk $[t_0,t_0+\Delta_0]$ shows that  the Burgers term functional and the Leray projection term functional and their firsr order derivatives have a Lipschitz constant $l_0$§ which is independent of $x$.
\begin{rem}
For (\ref{uppbound}) we may use the observation 
\begin{equation}
\begin{array}{ll}
\int_{t_0}^s\sup_{{\mathbb R}^n}{\big |}l(x-y) G^{\mu,r}_{\nu ,i}(\sigma,y){\big |}\\
\\
=\int_{t_0}^s\sup_{{\mathbb R}^n}{\big |}l(x-y)\frac{2y_i}{4\nu \sigma} G^{\mu,r}_{\nu }(s,y){\big |}\\
\\
\leq \int_{t_0}^s\sup_{{\mathbb R}^n,y_i\geq 0}{\big |}|l_x(-y)-l_x(-y^{-,i})|\frac{2y_i}{4\nu \sigma} G^{\mu,r}_{\nu }(s,y){\big |}\\
\\
\leq \int_{t_0}^s\sup_{{\mathbb R}^n,y_i\geq 0}{\big |}l_0\frac{4y^2_i}{4\nu \sigma} G^{\mu,r}_{\nu }(s,y){\big |}
\end{array}
\end{equation}
where $y^{i,-}=(y^{i,-}_1,\cdots,y^{i,-}_n)$ with $y^{i,-}_j=y_j$ for $j\neq i$ and $y^{i,-}_i=-y_i$.
\end{rem}
For a small time interval $[t_0,s]$ the complementary integral 
\begin{equation}\label{integrand}
\int_{t_0}^s\int_{{\mathbb R}^n\setminus B}{\big |}l\ast G^{\mu,r}_{\nu ,i}(s,y){\big |}dyds
\end{equation}
 becomes relatively small, i.e.,for every $\epsilon >0$ and  $\Delta_0=(s-t_0)$ small enough we have
\begin{equation}\label{releps}
\begin{array}{ll}
{\Big |}\int_{t_0}^s\int_{{\mathbb R}^n\setminus B}{\big |}l\ast G^{\mu,r}_{\nu ,i}(s,y){\big |}dyds{\Big |}\leq \epsilon l_0C(\mu_0\nu r^2)^{\delta} (s-t_0)^{1-\delta}\\
\\
=\epsilon l_0C(\mu_0\nu r^2)^{\delta} \Delta_0^{1-\delta}.
\end{array}
\end{equation}
This holds because the exponent of the Gaussian becomes small (note that the radius of $B$ does not depend on time - the estimates are designed for the case $\nu >0$). Recall that $l$ is Lipschitz such that  that the integrand of (\ref{integrand}) has an upper bound
\begin{equation}
\begin{array}{ll}
{\Big |}l\ast G^{\mu,r}_{\nu ,i}(s-t_0,y;0,0){\Big |}\leq 
c\frac{1}{\sqrt{s}^n}\exp\left(-\lambda\frac{|y|^2}{4\mu_0\nu r^2 (s-t_0)} \right),
\end{array}
\end{equation}
where $|y|^2\geq \mu_0\nu r^2$ and $\lambda >0$ is a finite positive constant which depends on $\mu_0>0$.
Hence if 
\begin{equation}
r\lesssim \sqrt{\Delta_0}^{1-\epsilon_0},~\mbox{for  $\epsilon_0>0$ small}
\end{equation}
then for any $\epsilon >0$ the relations in (\ref{uppbound}) and in (\ref{releps}) holds  for a time step size $\Delta_0>0$ small enough. This implies that solution increment due to the the nonlinear terms have an upper bound which is proportional to $\Delta_0^{2(1-\epsilon_0)\delta +1-\delta}$ where $\alpha_0:=2(1-\epsilon_0)\delta +1-\delta>1$ for $\epsilon_0$ small enough and given $\delta \in (0,1)$. This means that the increment $\delta v_i:=v_i-v_i(t_0,.)\ast_{sp}G_{\nu}$ ($\ast_{sp}$ denoting convolution with respect to the spatial variables) has an upper bound
\begin{equation}
\max_{1\leq i\leq n}\sup_{\sigma\in [t_0,t_0+\Delta_0]}|\delta v_i(\sigma,.)|_{H^m\cap C^m}\leq \Delta^{\alpha_0}
\end{equation}
with some $\alpha_0>1$ on a small time interval $\Delta_0$, a growth which is offset by the potential damping term over the same time interval if $\Delta_0$ is small enough. 

Next local contraction with respect to a $H^2\cap C^2$-norm implies that we have Lipschitz continuity of the Leray data function and its first order spatial derivatives. More precislely 
\begin{lem}
Given data $u^{r,t_0}_j(t_0,.)$ with ${\Big |}u^{r,t_0}_j(t_0,.){\Big |}_{H^2\cap C^2}=C_2<\infty$ there exists a time step size $\Delta_0$ such that $u^{r,t_0}_j(s,.)\in H^2\cap C^2$ for $s\in [t_0,t_0+\Delta_0]$. Moreover, the function
\begin{equation}
\begin{array}{ll}
y\rightarrow -\sum_j u^{r,t_0}_j(s,.) u^{r,t_0}_{i,j} (s,y)\\
\\
+ \sum_{j,r=1}^n\int_{{\mathbb R}^n}\left(K_{n,i}(y-z)\right)u^{r,t_0}_{j,r}(s,z) u^{r,t_0}_{r,j} (s,z)dz
\end{array}
\end{equation}
is in $C^1\cap H^1$.
\end{lem}
\begin{proof}
The first statement follows form local contraction of a standard iteration local solution scheme with respect to a $H^2\cap C^2$ norm (supremum over local time). Concerning the regularity of the  Leray data function we first observe that for  functions $u^{r,t_0}_{j}(s,.)\in C^2\cap H^2$ such that the first spatial derivatives are in $C^1$ and globally bounded. Hence we have
\begin{equation}\label{l2est}
 u^{r,t_0}_{j,r}(s,.) u^{r,t_0}_{r,j} (s,.)\in C^1\cap L^2\subset H^1_{loc}
\end{equation}
where $H^s_{loc}$ denotes the Sobolev space which is  locally $H^s$ for exponent $s\in {\mathbb R}$ (Sobolev $L^2$-theory). As data are in $H^1_{loc}$ we know from the regularity of uniformly elliptic operators of second order that $u^{r}_{t_0}(s,.)\in H^{1+2}_{loc}=H^3_{loc}$. Hence, $u^{r,t_0}_{j,r}(s,.)\in H^2_{loc}$. For dimension $n=2$ it follows that the right side is indeed in $H^2_{loc}$ (product rule for Sobolev spaces. Hence for data in $H^2\cap C^2$ we have indeed a regular Leray data function such that the first order spatial derivatives are Lipschitz.   

\end{proof}
Hence given a time horizon $T>0$ for the choice
\begin{equation}
r\sim  \frac{1}{1+T}
\end{equation}
and Lipschitz continuity (Lipschitz constant $l_0$) of the (first order spatial derivatives of)Euler-Leray data function, i.e., Burgers term operator plus Leray projection form operator applied to data $u^{r,t_0}_i(s,.)$, the nonlinear terms in (\ref{Navlerayusubschemepreia}) for $n\geq 3$ have an upper bound 
\begin{equation}
Cr\mu(s)(4\pi \mu(s)\nu r^2 )^{\delta}\lesssim \left( \frac{1}{1+T}\right)^{1+2\delta},~\delta \in \left(0,1 \right).
\end{equation}
 Note that the damping term (i.e. the potential term, or the second term on the right side of (\ref{Navlerayusubschemepreia})) has no parameter $r$ and the damping over a time interval $\Delta_0$ is
 of order $\frac{1}{T}\Delta_0$ times the data norm (for small $\Delta_0$).
More precisely, we have
\begin{lem}
Given a time horizon $T>0$ and ${\big |}u^{r,t_0}_i(0,.){\big |}_{H^m\cap C^m}\leq C_m$  there exists a constant $c(n,m)>0$ depending only on the dimension $n$ and the regularity order $m\geq 2$ and   a parameter 
\begin{equation}
r=\frac{1}{c(n,m)(C_m+1)^2(1+T)}
\end{equation} 
such that the representation in (\ref{Navlerayusubschemepreia}) holds on the time interval $\left[0,\frac{1}{\sqrt{3}}\right]$ for $0\leq |\beta|\leq m$ and we have for all $s\in \left[0,\frac{1}{\sqrt{3}}\right]$ 
\begin{equation}
{\big |}u^{r,t_0}_i(s,.){\big |}_{H^m\cap C^m}\leq {\big |}u^{r,t_0}_i(0,.){\big |}_{H^m\cap C^m}\leq C_m.
\end{equation}

\end{lem}
\begin{rem}
The constant $c(n,m)$ can be computed explicitly and contains upper bounds of local $L^1$ of $K_{,i}$ as factor of one of its summands etc.  
\end{rem}

Now let a time horizon $T>0$ be given, and assume that for the paramter $r>0$ of the preceding lemma
\begin{equation}
(1+t_0)C_m\geq (1+t_0){\big |}u^{r,t_0}_i(0,.){\big |}_{H^m\cap C^m}={\big |}v^{r,t_0}_i(t_0,.){\big |}_{H^m\cap C^m}
\end{equation}
has been proved up to some time $t_0\geq 0$. Then according to the Lemma above for this $r>0$ and  for all $s\in \left[0,\frac{1}{\sqrt{3}}\right]$ we have
\begin{equation}
{\big |}u^{r,t_0}_i(s,.){\big |}_{H^m\cap C^m}\leq {\big |}u^{r,t_0}_i(0,.){\big |}_{H^m\cap C^m}\leq C_m.
\end{equation}
Hence for $t\in [t_0,t_0+0.5]$ we have
\begin{equation}
\begin{array}{ll}
{\big |}v^{r,t_0}_i(t,.){\big |}_{H^m\cap C^m}\leq (1+t(s)){\big |}u^{r,t_0}_i(s,.){\big |}_{H^m\cap C^m}\\
\\
\leq (1+t){\big |}u^{r,t_0}_i(0,.){\big |}_{H^m\cap C^m}\leq (1+t)C_m.
\end{array}
\end{equation}
Hence,  for given time horizon $T>0$ there exists $r\sim \frac{1}{1+T}>0$ such that for all $0\leq t\leq T$
\begin{equation}
{\big |}v^{r,t_0}_i(t,.){\big |}_{H^m\cap C^m}\leq (1+t)C_m.
\end{equation}

We observe that this argument depends on the (mathematically) strong assumption $H^m\cap C^m,~m\geq 2$. The CKN-theory provides this situation for any small time, and it tells us about the asymptotic behaviour of singularity upper bounds near the $L^2$-data.

\end{document}